\documentclass[3p]{elsarticle}
\usepackage{amsmath,amssymb,amsthm}
\usepackage{lineno,hyperref}
\usepackage{tikz-cd} 
\usepackage{xcolor}
\usepackage{subcaption}

\newcommand{\FigPath }{}

\newcommand{\dimObs}{d}
\newcommand{\RKHS}{\mathcal{H}}
\newcommand{\Disc}{\mathcal{D}}

\newcommand{\Phimatrix}[1]{\Phi_{#1}}
\newcommand{\norm}{w}
\newcommand{\Orbit}{\mathcal{O}}
\newcommand{\DiscFour}{\mathcal{Z}}
\newcommand{\NGenFreq}{q}

\newcommand{\Fourier}{\mathcal{F}}
\newcommand{\Sinc}{\mathcal{S}}

\newcommand{\real}{\mathbb{R}}
\newcommand{\cmplx}{\mathbb{C}}
\newcommand{\integer}{\mathbb{Z}}
\newcommand{\num}{\mathbb{N}}

\newcommand{\TorusD}[1]{\mathbb{T}^{#1}}
\newcommand{\ratio}{r}
\newcommand{\unitary}{\mathcal{U}}
\DeclareMathOperator{\proj}{proj}
\DeclareMathOperator{\ran}{ran}
\DeclareMathOperator{\rank}{rank}
\DeclareMathOperator{\supp}{supp}
\DeclareMathOperator{\spn}{span}

\DeclareMathOperator{\Id}{Id}

\newcounter{enum_sav}

\newcommand{\blue}{\textcolor{black}}

\newtheorem{theorem}{Theorem}
\newtheorem{proposition}[theorem]{Proposition}
\newtheorem{lemma}[theorem]{Lemma}
\newdefinition{cor}[theorem]{Corollary}
\newdefinition{algorithm}{Algorithm}
\newdefinition{rk}[theorem]{Remark}

\newtheorem{Assumption}{Assumption}

\theoremstyle{remark}
\newtheorem*{rk*}{Remark}

\biboptions{sort&compress}

\begin{document}
\begin{frontmatter}

\title{Koopman spectra in reproducing kernel Hilbert spaces}

\author[]{Suddhasattwa Das\corref{mycorrespondingauthor}}
\ead{dass@cims.nyu.edu}
\cortext[mycorrespondingauthor]{Corresponding author}
\author[]{Dimitrios Giannakis}

\address{Courant Institute of Mathematical Sciences, New York University, New York, NY 10012, USA}

\begin{abstract}
Every invertible, measure-preserving dynamical system induces a Koopman operator, which is a linear, unitary evolution operator acting on the $L^2$ space of observables associated with the invariant measure. Koopman eigenfunctions represent the quasiperiodic, or non-mixing, component of the dynamics. The extraction of these eigenfunctions and their associated eigenfrequencies from a given time series is a non-trivial problem when the underlying system has a dense point spectrum, or a continuous spectrum behaving similarly to noise. This paper describes methods for identifying Koopman eigenfrequencies and eigenfunctions from a discretely sampled time series generated by such a system with unknown dynamics. Our main result gives necessary and sufficient conditions for a Fourier function, defined on $N$ states sampled along an orbit of the dynamics, to be extensible to a Koopman eigenfunction on the whole state space, lying in a reproducing kernel Hilbert space (RKHS). In particular, we show that such an extension exists if and only if the RKHS norm of the Fourier function does not diverge as $ N \to \infty $. In that case, the corresponding Fourier frequency is also a Koopman eigenfrequency, modulo a unique translate by a Nyquist frequency interval, and the RKHS extensions of the $N$-sample Fourier functions converge to a Koopman eigenfunction in RKHS norm. For Koopman eigenfunctions in $L^2$ that do not have RKHS representatives, the RKHS extensions of Fourier functions at the corresponding eigenfrequencies are shown to converge in $L^2$ norm. Numerical experiments on mixed-spectrum systems with weak periodic components demonstrate that this approach has significantly higher skill in identifying Koopman eigenfrequencies compared to conventional spectral estimation techniques based on the discrete Fourier transform.
\end{abstract}


\end{frontmatter}
\modulolinenumbers[5]
\section{Introduction}\label{sect:intro}

A common scenario in the analysis of data generated by dynamical systems is that the underlying discrete- or continuous-time flow is unknown, and the system is observed through some observation map $F$ taking values in a vector space (the data space). The challenge is then to infer various properties of the system from the time series $\{F(x_0), F(x_1),\ldots\}$, where \{$x_0, x_1,\ldots$\} is an orbit in the system's state space with a fixed sampling interval $ \Delta t $. The focus of this paper is on the identification of eigenfunctions of an operator called the Koopman operator \cite{Koopman31}, which governs the evolution of observables under the dynamics. We will assume throughout that the dynamics is measure-preserving and ergodic.

In this setting, Koopman eigenfunctions in the $L^2$ space associated with the invariant measure form a distinguished class of observables that evolve by multiplication by a time-periodic factor $e^{i\omega t}$, $\omega\in\mathbb{R}$, even if the dynamics is aperiodic. They extract temporally coherent, and thus highly predictable, temporal patterns from complex dynamics, having high physical interpretability by virtue of being associated with an operator intrinsic to the dynamical system generating the data. In addition, observables lying in the span of these eigenfunctions have an integrable time evolution, and are useful for reduced order modeling and forecasting. Due to these properties, Koopman eigenfunctions warrant identification from data. 

There have been many approaches to the identification of Koopman spectra (and the spectra of the related Peron-Frobenius operators, which are duals to Koopman operators), such as methods based on state space partitions \cite{DellnitzJunge99}, harmonic averaging \cite{MezicBanaszuk04,Mezic05}, Krylov subspace iteration \cite{Schmid10,RowleyEtAl09}, dictionary-based approximation \cite{TuEtAl14,WilliamsEtAl15,KutzEtAl16}, Galerkin approximation \cite{GiannakisEtAl2015,Giannakis17}, delay-coordinate embeddings \cite{BruntonEtAl17,GiannakisEtAl2015,Giannakis17,DasGiannakis_delay_Koop}, and spectral moment estimation \cite{KordaEtAl2018}. Among these, the methods based on harmonic averaging are closely related to spectral estimation techniques via the discrete Fourier transform (DFT). Given the time series \{$F(x_0), F(x_1),\ldots$\}, harmonic averaging techniques compute the quantities 
\begin{equation}\label{eqn:FFT}
\Fourier_{\omega,N}F := \frac{1}{N}\sum_{n=0}^{N-1} e^{-i\omega n\, \Delta t} F(x_n) 
\end{equation}
for a set of candidate frequencies $\omega \in \real$. If a frequency $ \omega $ in the candidate set lies in the point spectrum of the Koopman group, then $f$ acquires a discrete spectral component $a_\omega $ at that frequency, and as the sample size $N$ increases, $ \Fourier_{\omega,N}f $ converges to $ a_\omega $. Thus, the quantities in~\eqref{eqn:FFT} can in principle reveal the point spectrum of the dynamics based on empirical time-ordered measurements of observables. 

Despite this useful property, a direct application of harmonic averaging has a number of limitations. For example, signals are often noisy, and even in a deterministic system without noise, if there is a non-empty \emph{continuous spectrum} [see the definitions preceding \eqref{eqn:L2_decomp}], then the signal will be spectrally similar to one generated by a noisy source. In such cases, it may be difficult to distinguish the true discrete spectral components from those due to noise and/or the continuous spectrum. Moreover, the magnitude of the discrete spectral components carried by the signal may rapidly decay with increasing frequency, making the task more difficult, and may even vanish if $F$ is orthogonal to the corresponding Koopman eigenspace. Figure~\ref{fig:l63_skew_additive_FFT} illustrates some of the shortcomings of spectral estimation via~\eqref{eqn:FFT} with an application to a chaotic signal.

\begin{figure}
\centering
\captionsetup{width=\linewidth}\label{subfig:1a}
\includegraphics[width=\textwidth]{\FigPath 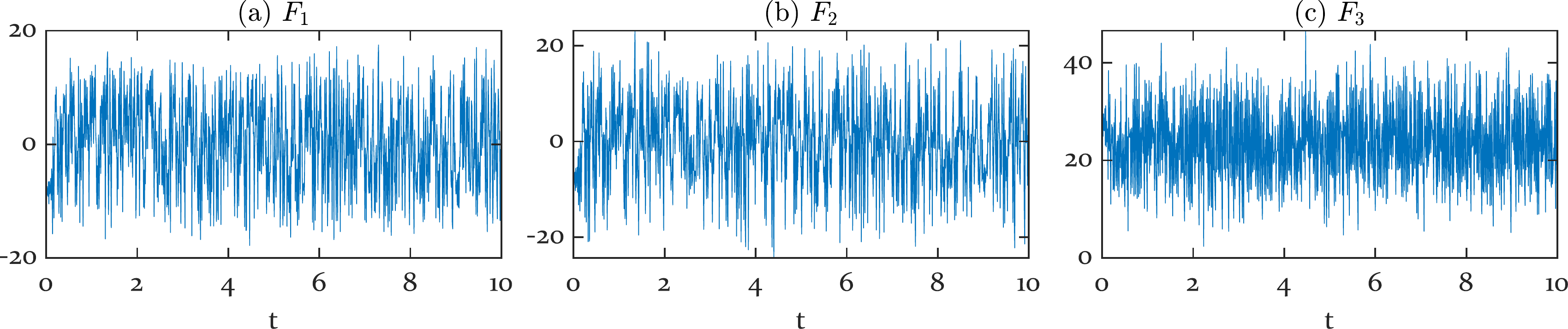}

\includegraphics[width=\textwidth]{\FigPath 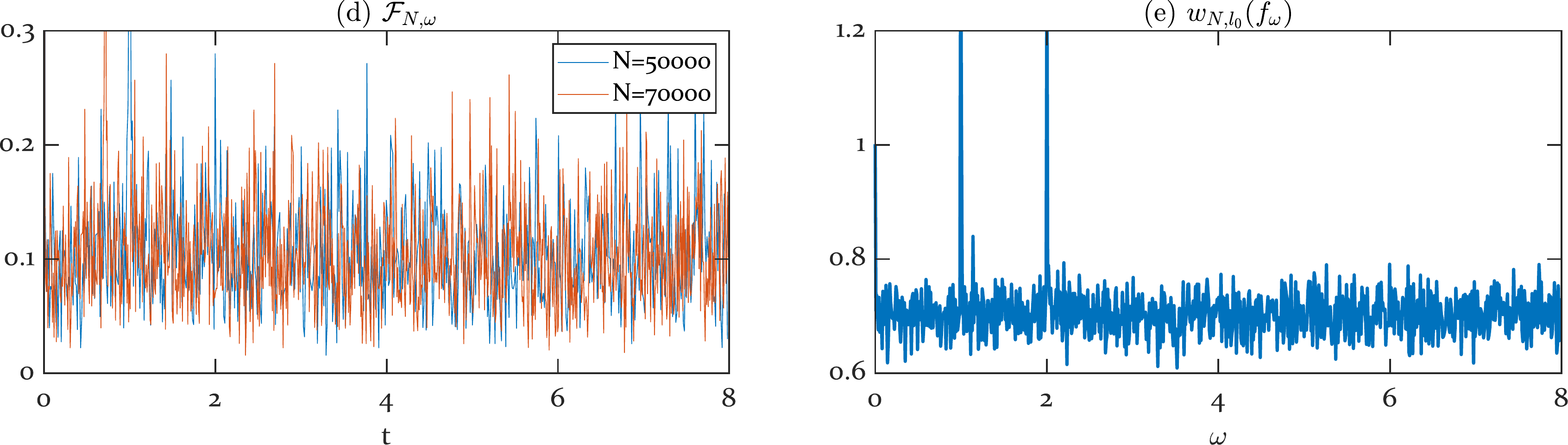}
\caption[results]{ A mixed-spectrum signal (a,b,c) and the results of spectral analysis via the DFT (d) and the RKHS-based approach using a Markov-normalized Gaussian kernel (e). The input signal $F(t) = (F_1(t),F_2(t), F_3(t) ) \in \real^3$ is a superposition of a periodic signal and a chaotic signal, generated by the product dynamical system in \eqref{eqn:l63_skew_additive}, and has both discrete and continuous spectral components. The Koopman eigenfrequencies present in $F(t) $ are 1 and 2. The DFT-based spectral analysis was performed on data sets consisting of $N=\text{50,000}$ and $\text{70,000}$ samples. In either case, it fails to provide a clear demarcation of the true frequencies in the point spectrum of the system. In particular, note that the $\omega = 1 $ peak in the DFT spectrum for $N=\text{50,000}$ disappears at $N=\text{70,000}$. On the other hand, the RKHS-based approach described in this paper correctly identifies the $ \omega = 1 , 2$ eigenfrequencies of the Koopman operator. The method, applied here to the same 70,000-sample dataset as in (d), constructs a data-driven eigenbasis of an RKHS $ \RKHS $ of functions on state space, whose elements can be nonlinear functions of the input data (depending on the kernel employed). Then, for each Fourier function $ f_\omega $, sampled at $N$ time-ordered points along a dynamical orbit, it computes the squared RKHS norm $w_{N,l}(f_\omega) $ of its projection onto the RKHS subspace spanned by the leading $ l $ basis functions (here, $l=1240$). Theorem~\ref{thm:D} shows that as $N\to\infty$, and for sufficiently large fixed $l$, this quantity converges to a nonzero finite number if and only if a translate $\omega + q 2\pi/\Delta t$ with $ q \in \mathbb{Z}$ is a Koopman eigenfrequency, and otherwise converges to 0. }
\label{fig:l63_skew_additive_FFT}
\end{figure} 

In this work, we approach the problem of estimating the point spectra of Koopman operators as an extrapolation problem. Specifically, we seek to extend a candidate eigenfunction from its values on the sample trajectory to the entire space, in a reproducing kernel Hilbert space (RKHS) of functions. Our approach is based on the observation that along an orbit of the dynamics a continuous Koopman eigenfunction of a measure-preserving dynamical system behaves like a Fourier function, evolving as $e^{i\omega t}$ for a real frequency $\omega$. On the other hand, it is not the case that every Fourier function on an orbit extends to a Koopman eigenfunction lying in a space of observables of sufficient regularity. Choosing the class of RKHSs as Hilbert spaces naturally encapsulating a notion of regularity of observables, we will show that the Fourier functions on orbits admitting RKHS extensions to the entire state space are precisely those having RKHS representatives constructed from $N$ samples on the orbit, with convergent squared RKHS norm as $ N \to \infty $. Moreover these extensions will be shown to be Koopman eigenfunctions with eigenfrequencies equal to the corresponding Fourier frequencies, modulo an irreducible ambiguity (aliasing) due to discrete-time sampling. The RKHS framework also allows for stable evaluation of the approximate Koopman eigenfunctions determined from $N $ samples at arbitrary points on state space, in contrast to approximation in $L^2$ spaces which only yields estimates for the eigenfunction values at the sampled states. The action of the Koopman operator on RKHSs was also considered in \cite{Kawahara2016,KlusEtAl2017}, although these studies rely on the strong assumption that the RKHS is invariant under the Koopman group.

In Section \ref{sect:examples}, we will revisit the example shown in Figure~\ref{fig:l63_skew_additive_FFT}, and discuss how the application of our methods using RKHSs associated with covariance kernels makes the analysis related to harmonic averaging. As shown in Figure~\ref{fig:l63_skew_additive_FFT}, the RKHS-based analysis utilizing a Gaussian kernel, which results in an infinite-dimensional RKHS despite the fact that the data measurement function takes values in a finite-dimensional space, $\real^3$, correctly identifies two Koopman eigenfrequencies, which in this case is enough to recover all eigenfrequencies. 

\section{Assumptions and statement of the main results}\label{sec:intro}

As previously stated, we are interested in the spectral analysis of continuous-time, measure-preserving, ergodic flows, and the basic assumptions on the system are stated below. 

\begin{Assumption}\label{assump:A1}
$\Phi^t:M\to M$, $ t \in \real $, is a continuous flow on a metric space $M$, possessing an invariant, ergodic Borel probability measure $\mu$, whose support $\supp \mu$ is a compact set $X$, which is not a fixed point of the dynamics. The system is sampled at a fixed interval $ \Delta t > 0 $, such that $ \mu $ is an ergodic measure for the discrete-time evolution map $ \Phi^{n\,\Delta t} : M \to M $, $ n \in \integer $.
\end{Assumption}

Given an initial point $x_0 \in M$, we let $\Orbit$ denote the orbit $ \{x_n=\Phi^{n\Delta t}(x_0) : n \in \integer\} \subset M$, and $X_N $ the finite trajectory $\{x_0,\ldots,x_{N-1}\} \subset \mathcal{O}$. As will be shown in Lemma~\ref{lem:ergdc_samplng} below, by the restriction on the sampling interval $\Delta t$ in Assumption~\ref{assump:A1}, $\mathcal{O}$ is a dense subset of $X$ for $\mu$-a.e.\ starting point $x_0$. This, in conjunction with the fact that $\supp \mu$ is not a fixed point, implies that $\mathcal{O}$ is an infinite set $\mu$-a.s., only containing distinct points. The latter will be implicitly assumed henceforth.


\paragraph{The Koopman operator} Under Assumption \ref{assump:A1}, a natural space of observables of the dynamical system is the Hilbert space $L^2(\mu)$ of equivalence classes of complex-valued functions on $X$, square-integrable with respect to $\mu$. This space is a Hilbert space, equipped with the inner product $\langle f, g\rangle_{\mu} = \int_X f^* g \, d\mu$. The dynamical flow $\Phi^t$ induces a linear map $U^t$ on the vector space of complex-valued functions on $M$ (and $X$), defined as $ U^t f = f\circ \Phi^t $. $U^t$ is called the Koopman operator at time $t$, and it acts on observables by composition with the flow map. Its action induces a norm-preserving operator on the Banach space of continuous, complex-valued functions on functions on $X$, denoted $C^0(X)$. Furthermore, this action extends to the spaces $L^p(\mu)$ for every $p\geq 1$. In fact $U^t : L^2(\mu) \to L^2(\mu) $, $ t \in \real $, is a strongly continuous, 1-parameter group of unitary operators \cite{EisnerEtAl15}. The unitarity of $U^t$ stems from the fact that $\Phi^t$ is an invertible, $\mu$-preserving map, and it implies that all of its eigenvalues lie on the unit circle of the complex plane. In addition, it follows from the continuity of the map $ t \mapsto \Phi^t$ that every eigenvalue of $U^t $ has the form $e^{i\omega t}$, where $\omega$ is a real eigenfrequency. As a result, an eigenfunction $z \in L^2(\mu)$ corresponding to that eigenvalue satisfies the equation 
\begin{equation}\label{eqn:Def_koop_eigen}
U^tz=e^{i\omega t}z.
\end{equation}
In fact, it follows from Stone's theorem for strongly continuous unitary groups \cite{Stone1932} that $ i \omega $ is an eigenvalue of the generator $ V : D(V) \to L^2(\mu) $ of the Koopman group; a skew-adjoint unbounded operator with a dense domain $D(V) \subseteq L^2(\mu)$, acting on observables as a ``time derivative'', viz.\ $ V f = \lim_{t\to0} ( U^t f - f ) / t $. By ergodicity of the flow, and under Assumption~\ref{assump:A1}, all eigenvalues of $ V $ and $U^{\Delta t}$ are simple. 

\paragraph{An extrapolation problem} For each frequency $\omega \in \real$, we define the function 
\begin{equation}\label{eqn:def:f_omega}
f_{\omega}:\Orbit\to\cmplx, \quad f_{\omega}(x_n) = e^{in\omega\,\Delta t}, 
\end{equation}
and seek to determine whether $f_\omega$ extends to a function $ \bar f_\omega : X \to \cmplx$ of appropriate regularity. Observe, in particular, that if $\omega$ were a Koopman eigenfrequency corresponding to a continuous eigenfunction $z_\omega$ such that $z_\omega(x_0)=1$, then by \eqref{eqn:Def_koop_eigen}, the restriction $z_\omega|\Orbit$ to the orbit coincides with $f_{\omega}$ in \eqref{eqn:def:f_omega}, and thus with $f_\omega| X_N$ (i.e., the restriction of $ f_\omega $ on the finite trajectory $X_N$). Here, we are interested in this question in the reverse direction, i.e., our objective is to identify for which $\omega \in \real$ the Fourier function $f_{\omega} $ can be extended to a Koopman eigenfunction, satisfying~\eqref{eqn:Def_koop_eigen} at eigenfrequency $ \omega$. This is essentially an extrapolation problem of a candidate function from a countable set $\Orbit$ to the entire space $X$. In addition, we are interested in performing this extrapolation in a data-driven manner; that is, from measurements of the system taken on the finite trajectory $X_N$ without prior knowledge of the dynamics or the structure of state space.

\paragraph{Aliasing} As with any empirical signal processing technique operating in a discrete-time sampling environment, our ability to estimate frequencies is limited by Nyquist sampling criteria. Observe, in particular, that $f_\omega$ and $f_{\omega'}$ from~\eqref{eqn:def:f_omega} are identical whenever $\omega$ and $\omega'$ differ by an integer multiple of $2\pi/{\Delta t}$. This leads to the well known fact that the distinct Fourier functions on the orbit $\mathcal{O}$ can be parameterized by frequencies $\omega$ in the interval $ [ - \pi / \Delta t, \pi / \Delta t )$. This degeneracy will have implications in the results stated in Theorems~\ref{thm:A} and~\ref{thm:D} below. For example, it will preclude us from uniquely inferring eigenfrequencies of the generator $V$ from eigenvalues of the discrete-time Koopman operator $U^{\Delta t}$. 

\paragraph{Reproducing kernel Hilbert spaces} While natural from a theoretical point of view, the Banach space of continuous functions on $X$ is arguably not well-suited for meeting the objectives set forth above, for it lacks the inner product structure that significantly facilitates the implementation of data-driven techniques. Instead, following the widely adopted paradigm in statistical learning theory \cite{CuckerSmale01}, we will focus on identification of Koopman eigenfunctions lying in an RKHS, which, in addition to the pointwise approximation guarantees provided by the uniform norm on $C^0(X)$, its Hilbert space structure allows the construction of data-driven algorithms based on standard linear algebra tools. Throughout this work, we will restrict attention to RKHSs with continuous reproducing kernels, so that convergence in RKHS norm implies convergence in $C^0(X)$ norm. Before proceeding, we briefly review some of the main properties of RKHSs. These concepts will be discussed in more detail in Section~\ref{sect:theory_rkhs}.

An RKHS $\RKHS$ on $M$ is a Hilbert subspace of the linear space of complex-valued functions $f:M\to\cmplx$, equipped with an inner product $ \langle \cdot, \cdot \rangle_{\RKHS}$, such that for every $x\in M$, the point-evaluation map $\delta_x:\RKHS\to\cmplx$, $\delta_x f = f( x )$, is a bounded, and thus continuous, linear functional. The Hilbert space structure of $\RKHS$ allows operations such as orthogonal projections, needed for many numerical procedures. Moreover, because $\delta_x $ is bounded for every $x\in M$, convergence in $\RKHS$ norm implies pointwise convergence on $M$ (in fact, under the assumptions made below, this convergence is uniform on compact sets, including $X$). Thus, RKHSs combine useful properties of both $L^2(\mu)$ (Hilbert space structure) and $C^0(X)$ (pointwise evaluation by bounded functionals).

By the Moore-Aronszajn theorem \cite{Aronszajn1950}, every RKHS is uniquely determined through its reproducing kernel; a bivariate function $k : M \times M \to \cmplx$ with the following properties:
\begin{enumerate}[(i)]
\item $ k $ is conjugate symmetric, i.e., $k(x,y) = k(y,x)^*$ for all $x,y \in M$.
\item $ k$ is positive-definite, i.e., for every $x_1,\ldots,x_n\in M$ and $a_1,\ldots,a_n\in\cmplx$, $\sum_{i,j=1}^{n} a^*_i a_j k(x_i,x_j)\geq 0$.
\item For every $x\in M$, the kernel sections $k(x,\cdot)$ lie in $\RKHS$.
\item The reproducing property $ f( x ) = \delta_x f = \langle k(x,\cdot), f \rangle_{\RKHS} $ holds for every $f\in \RKHS$ and $ x \in M $. 
\end{enumerate}
These properties imply that $\langle k( x, \cdot ), k( y, \cdot ) \rangle_{\RKHS} = k(x,y) $ for all $x,y \in M$, and $\RKHS$ is the closure (in the $\RKHS$ norm) of finite sums $f=\sum_{i=1}^n a_i k(x_i,\cdot)$ with $ a_1, \ldots, a_n \in \cmplx $ and $ x_1, \ldots, x_n \in M$. The kernel $k$ is said to be \emph{strictly} positive-definite if the inequality in property (ii) above is strict whenever the points $ x_1, \ldots, x_n $ are distinct, and at least one of the coefficients $a_1, \ldots, a_n $ is nonzero. 

Let $\RKHS(X)$ denote the RKHS on $X$ with $ k| X\times X $ as its reproducing kernel, and note that this space embeds naturally and isometrically into $\RKHS$, so we may view it as a closed subspace of the latter space. If $ k| X \times X $ is strictly positive-definite and continuous, then $\RKHS(X)$ is a dense subspace of $C^0(X)$ (a property often referred to as universality \cite{MicchelliEtAl2006}), and convergence in $\RKHS(X)$ norm implies convergence in $C^0(X)$ norm. If $k$ is continuous on the whole of $M\times M$, then $\RKHS$ is a subspace of $C^0(M)$, the space of continuous functions on $M$ \cite{FerreiraMenegatto2013}. In fact, if $k $ is a smooth kernel on a compact manifold, $\RKHS$ is a subset of the $C^r$ spaces on the manifold for all $r \geq 0 $. For our purposes, RKHSs have the key property that they allow \emph{out-of-sample extensions} of $L^2(\nu) $ equivalence classes of functions with respect to any finite, compactly supported Borel measure $\nu$ on $M$ to everywhere-defined functions in $\RKHS$. 
In light of the above, we will require that the following conditions on kernels be satisfied:
\begin{Assumption}\label{assump:A2}
$k : M \times M \to \real$ is a reproducing kernel for an RKHS $ \RKHS $. Moreover, the restricted kernel $k|X\times X$ is (i) continuous; and (ii) strictly positive-definite. 
\end{Assumption}

Consider now the finite trajectory $X_N$ and the probability measure $\mu_{N} = \sum_{n=0}^{N-1} \delta_{x_n}/N$, where $\delta_{x_n}$ is the Dirac delta-measure supported on the point $x_n$. $\mu_{N}$ is called the \emph{sampling measure} supported on the finite trajectory $X_N \subseteq \mathcal{O} $. By ergodicity, for $\mu$-a.e.\ $x_0 \in M$, as $N\to\infty$, $\mu_N$ converges weakly to the invariant measure $\mu$, i.e., for every continuous function $f:M\to\cmplx$, $\lim_{N\to\infty} \int_M f\, d\mu_N= \int_M f\, d\mu$. As a data-driven analog of $L^2(\mu)$, we employ the Hilbert space $L^2(\mu_N) $, consisting of equivalence classes of functions on $X$ taking the same values on the finite set $X_N$, while taking arbitrary values on $M\setminus X_N$. Because all points in the orbit $\mathcal{O}$ are distinct (see Assumption~\ref{assump:A1}), $L^2(\mu_N)$ is an $N$-dimensional Hilbert space isomorphic to $\cmplx^N$, equipped with a normalized Euclidean inner product, $\vec f \cdot \vec g / N$. It is also isomorphic to the Hilbert space of complex-valued functions on $X_N$, equipped with the normalized Euclidean product. In what follows, we will identify elements of the latter space with elements of $L^2(\mu_N)$ without the use of additional notation.

Next, as a data-driven analog of $\RKHS$, we consider the $N$-dimensional subspace $\RKHS_N \subset \RKHS $, defined as the linear span of the set of kernel sections $\{k({x_0},\cdot), \ldots, k(x_{N-1},\cdot) \} $ on $X_N$. We will show in Section~\ref{sect:theory_rkhs} that the union of the spaces $\RKHS_N | X$ is dense in $\RKHS(X)$, using the fact that $\cup_{N\in\num}X_N = \Orbit$ and Assumption~\ref{assump:A2}(i) . 

\paragraph{Fitting RKHS functions} As will be discussed in more detail in Section~\ref{sect:theory_rkhs}, under Assumption~\ref{assump:A2}(ii), there exists an extension operator $ T_N : L^2(\mu_N) \to \RKHS $, mapping the equivalence class $ f \in L^2(\mu_N) $ to a pointwise-defined function $ h = T_N f \in \RKHS$, such that $f(x_n)= h( x_n ) $ for all $ x_n \in X_N $. Moreover, $ h|X $ is the unique element of $\RKHS(X)$ with this property. In fact, $h$ lies in $ \RKHS_N$, i.e., it is equal to a linear combination of $k({x_0},\cdot),\ldots, k(x_{N-1},\cdot)$. Using this operator, we define the non-negative functional
\begin{equation}\label{eqn:norm}
\norm_N : L^2(\mu_N) \to \real, \quad \norm_N (f) = \left\|T_Nf\right\|_{\RKHS}^2.
\end{equation}
The functional $w_N$ induces a norm on equivalence classes of functions with respect to $ \mu_N $, distinct from the $L^2(\mu_N)$ norm. Intuitively, the ratio 
\begin{equation}\label{eqRKHSRat}
r_N( f ) = \norm_N( f )/ \lVert f \rVert_{\mu_N}
\end{equation}
can be thought of as a measure of ``roughness'' of $f$ analogous to a Dirichlet energy; that is, the larger that quantity is, the stronger the degree of spatial variability of its RKHS extension $ T_N f $ becomes. If we have a function $f:\mathcal{O}\to \cmplx$ defined on the entire orbit (e.g., $f_\omega$ from~\eqref{eqn:def:f_omega}), then for brevity of notation, we will abbreviate $w_N(f|X_N)$ and $ \lVert f | X_N \rVert_{\mu_N}$ by $w_N(f)$ and $\lVert f \rVert_{\mu_N}$, respectively. Our main result below establishes a necessary and sufficient condition in terms of $\norm_{N}(f_{\omega})$ for $\omega$ to be a Koopman eigenfrequency corresponding to a Koopman eigenfunction in $\RKHS$, modulo a unique translate by a Nyquist frequency interval.

\begin{theorem}\label{thm:A}
\blue{Under Assumptions~\ref{assump:A1} and~\ref{assump:A2} the following holds for $\mu$-a.e.\ $x_0\in M$}: For every $\omega \in \real$, let $f_{\omega}|X_N$ be a Fourier function on the finite trajectory $ X_N $ as in \eqref{eqn:def:f_omega}. Then, 
\begin{enumerate}[(i)]
\item $\lim_{N\to\infty}\norm_N (f_\omega)=\infty$ iff $f_{\omega}$ does not have an extension $\bar{f}_{\omega}\in\RKHS$.
\setcounter{enum_sav}{\value{enumi}}
\end{enumerate}
Moreover, if $\bar{f}_{\omega}$ exists: 
\begin{enumerate}[(i)]
\setcounter{enumi}{\value{enum_sav}}
\item $\lim_{N\to\infty}\norm_N (f_\omega) = \|\bar{f}_{\omega}\|_{\RKHS}^2$. 
\item $\bar{f}_{\omega}$ is an eigenfunction of the Koopman operator $U^{\Delta t}$ and the generator $V$, with corresponding eigenvalue and eigenfrequency $e^{i\omega\Delta t}$ and $\omega + 2\pi q /\Delta t$, respectively, for a unique $q\in\integer$.
\end{enumerate}
\end{theorem}

\begin{rk*} The elements of the sequence $\norm_N (f_\omega)$ in Theorem \ref{thm:A} are the squared norms of the vectors $T_N(f_\omega|X_N)$ lying in $\RKHS$. For a general sequence of vectors in a Hilbert space with bounded norm, the norms may not converge, and even if they did, the vectors themselves may not be convergent. For the vectors $ f_\omega|X_N$, however, the boundedness of their norm is in fact equivalent to them forming a convergent sequence.
\end{rk*}

Theorem \ref{thm:A} gives necessary and sufficient conditions under which a Fourier function sampled on a countable set can be extended to a Koopman eigenfunction in an RKHS. In other words, it provides a characterization of the following set of frequencies:
\begin{equation}\label{eqn:def:Omega}
\Omega := \{ \omega\in\real : \text{$\omega$ is a Koopman eigenfrequency corresponding to a Koopman eigenfunction in $\RKHS$} \}.
\end{equation}
Theorem \ref{thm:A} establishes a connection between the spectral properties of the dynamics, determined solely by the Koopman group $ \{ U^t \} $, and the RKHS, determined solely by the kernel $k$. In effect, because $ \lVert f_\omega \rVert_{\mu_N} = 1 $ for every frequency $ \omega \in \real $ and all $ N \in \num $, this connection is realized through the limiting behavior of the RKHS-induced roughness measure from~\eqref{eqRKHSRat} of Fourier functions on dynamical trajectories of increasing size. That is, Koopman eigenfunctions in $L^2(\mu)$ with representatives in $\RKHS$ can be characterized as $\RKHS$-extensions of Fourier functions $f_\omega | X_N $ in $L^2(\mu_N)$ with non-divergent roughness $r_N(f_\omega)$ as $N \to \infty$. 

It is important to note that the result in Theorem~\ref{thm:A}(iii) does not provide provide direct knowledge of the Koopman eigenfrequency underlying a Fourier function $ f_\omega $, but establishes that there is a unique translate of the form $ \omega + 2 \pi q / \Delta t $, $ q \in \integer $, which \emph{is} an eigenfrequency. As stated above, this ambiguity stems from the fact that any two frequencies $ \omega$ and $ \omega' $ that differ by an integer multiple of $ 2\pi/\Delta t $ lead to the same Fourier unction $ f_\omega $ on the orbit $\mathcal{O}$ associated with the sampling interval $ \Delta t $. In a practical scenario, this type of ambiguity can be resolved if one has access to an additional sampling of the dynamical system, taken at a rationally independent interval from $ \Delta t $. Specifically, we have:

\begin{cor}\label{corFreq}
Let $ \Delta t'>0 $ be a sampling interval such that $ \mu $ is an ergodic measure for the discrete-time map $ \Phi^{n\,\Delta t'} : M \to M $, and $\Delta t/ \Delta t' $ is an irrational number. Let also $\mathcal{O}' = \{ \Phi^{n \, \Delta t'}(x'_0) : n \in \integer \} $ be the discrete-time orbit at the sampling interval $ \Delta t'$, starting from a point $x'_0 \in M $. Then, every frequency $\omega \in \real$ such that the corresponding Fourier functions on $\mathcal{O}$ and $\mathcal{O}' $ simultaneously have $\mathcal{H}$ extensions is a Koopman eigenfrequency.
\end{cor}

\begin{rk*} Corollary~\ref{corFreq} is a direct consequence of Lemma~\ref{lem:ergdc_samplng}(iii) below, which implies that $ e^{i\omega \, \Delta t}$ and $ e^{i\omega \, \Delta t'}$ are eigenvalues of $U^{\Delta t}$ and $U^{\Delta t'}$, respectively, iff $\omega$ is an eigenfrequency of the generator $V$. Here, our main interest is in spectral estimation for the Koopman group from individual dynamical trajectories, so we will not pursue a numerical application of the corollary. However, one can certainly envision experimental scenarios where it is possible to control the sampling interval, and in such scenarios Corollary~\ref{corFreq} provides an anti-aliasing tool to resolve Koopman eigenfrequencies of a continuous-time system from discretely sampled time series.
\end{rk*}

Theorem \ref{thm:A} is in fact a consequence of Theorem \ref{thm:B} below, which is a general RKHS result that only depends on the dynamical orbit $\mathcal{O}$ lying dense in $X$. 
\begin{theorem}\label{thm:B}
Let $k:M \times M \to \real$ be a kernel satisfying Assumption \ref{assump:A2}, $\RKHS$ the corresponding RKHS, and $\Orbit\subset X$ a dense, countable set containing distinct points. Denoting the first $N$ points of $\mathcal{O}$ by $X_N$, the following hold: 
\begin{enumerate}[(i)] 
\item For every $f\in \RKHS(X)$, $\lim_{N\to\infty}\|T_N( f|X_N ) \|_{\RKHS}$ = $\|f\|_{\RKHS}$.
\item If $f:\Orbit \to\cmplx$ is such that $\|T_N\left( f|X_N \right) \|_{\RKHS}$ does not diverge as $N\to\infty$, then $f$ has a unique extension in $\RKHS(X)$.
\end{enumerate}
\end{theorem}

\begin{rk*} Theorem \ref{thm:B} holds in a general RKHS whose corresponding reproducing kernel satisfies Assumption~\ref{assump:A2}, and does not require an underlying time-flow or dynamics. The result is therefore of broader applicability than dynamical systems, as it gives necessary and sufficient conditions for functions on dense, countable sets to have RKHS extensions. We have stated it as one of the main results because we have found no similar result in the literature.
\end{rk*}

\paragraph{Spectral decomposition} Let $\Disc$ be the closed subspace of $L^2(\mu)$ spanned by the eigenfunctions of $U^t$, and $\Disc^\bot$ its orthogonal complement. Systems in which $\Disc$ contains non-constant functions and $\Disc^\bot$ is nonzero are called mixed-spectrum systems. The $L^2(\mu)$ space of a general measure-preserving system admits the $ U^t$-invariant decomposition
\begin{equation} \label{eqn:L2_decomp}
L^2(X,\mu)=\Disc\oplus\Disc^\bot.
\end{equation}
In the spectral study of dynamical systems, it is a classical approach to study the dynamics separately on $\Disc$ and $\Disc^\bot$; see, e.g., \cite{Halmos1956}. This is because not only are these spaces invariant under $U^t$, they also represent the quasiperiodic and weak-mixing (chaotic) component of the underlying dynamics \cite{DSSY_Mes_QuasiP_2016,DasJim2017_SuperC}. In particular, every observable $ f = \sum_j c_j z_j \in \mathcal{D} $ can be expanded in an orthonormal basis $\{ z_j \} $ consisting of Koopman eigenfunctions, and thus has integrable (quasiperiodic) time evolution, $ U^t f = \sum_j e^{i\omega_jt} c_j z_j$. On the other hand, observables $ g \in \mathcal{D}^\perp $ have an expansion associated with the continuous spectrum of $U^t$ and exhibit a weak-mixing property, $\lim_{t\to\infty} t^{-1} \int_0^t \lvert \langle h, U^s g \rangle_{\mu} \rvert \, ds =0 $, for all $ h \in L^2(\mu)$, characteristic of chaotic evolution. In a data-driven setting, the invariant splitting in~\eqref{eqn:L2_decomp} was introduced in \cite{Mezic05} in the context of harmonic averaging techniques, and was also employed in \cite{DasGiannakis_delay_Koop} in Galerkin approximation techniques for the eigenvalues and eigenfunctions of the generator of the Koopman group.

High-dimensional complex systems are typically of mixed spectrum, and for such systems an important task at hand is to identify the Koopman eigenbasis of $\mathcal{D}$ and the associated eigenfrequencies from data. The following result establishes data-driven criteria to determine whether a candidate frequency $ \omega \in \real $ is a Koopman eigenfrequency, while also providing an alternative characterization of the set of eigenfrequencies $\Omega$ from~\eqref{eqn:def:Omega} corresponding to RKHS-extensible eigenfunctions to that established in Theorem~\ref{thm:A}. We will employ a spectrally truncated analog $T_{N,l}$ of $T_N$, which performs RKHS extension after projection onto an $l$-dimensional subspace of $L^2(\mu_N)$ spanned by a collection of leading kernel eigenfunctions. See \eqref{eqn:def:TNl} ahead for an explicit definition of this operator.

\begin{theorem}\label{thm:D}
Let Assumptions \ref{assump:A1} and \ref{assump:A2} hold, and $\norm_N$ be as in \eqref{eqn:trunc}. Then, for every $N\in\num$, there exists a sequence of approximations $\left( \norm_{N,l} \right)_{l=1}^{N}$ of $\norm_N$ such that the following hold for $\mu$-a.e.\ $x_0$: 
\begin{enumerate}[(i)]
\item If $\omega \in \real$ is such that for no integer $q$ is $\omega + 2\pi q /\Delta t$ a Koopman eigenfrequency, then 
\[ \lim_{N\to\infty} \norm_{N,l}(f_{\omega})=0, \quad \forall l\in\num.\]
\item Conversely, \blue{there exists a sequence of integers $l_1 < l_2 < l_3 < \ldots$, depending only on the kernel $k$, such that } if $\omega + 2\pi q /\Delta t$ is an eigenfrequency for some $q\in\integer$, then there is an $ \epsilon > 0 $ such that for every $m\in\num$, $\lim_{N\to\infty} \norm_{N,l_m}(f_{\omega})$ exists and is greater than $\epsilon$ for $m$ large-enough. Moreover, there exists a Koopman eigenfunction $z$ corresponding to $\omega + 2\pi q/\Delta t$ such that
\[\lim_{l\to \infty} \lim_{N\to\infty} \left\| T_{N,l}(f_\omega) - z \right\|_{\mu} = 0.\]
\item If $\omega$ is a Koopman eigenfrequency in $\Omega$, then by Theorem~\ref{thm:A}, there exists an RKHS extension $\bar{f}_{\omega}$ of $f_{\omega}$, and 
\[\lim_{l\to\infty}\lim_{N\to\infty} \norm_{N,l}(f_{\omega}) = \lim_{\substack{l,N\to\infty,\\ l\leq N}} \norm_{N,l}(f_{\omega}) = \lim_{N\to\infty} \norm_{N}(f_{\omega}) = \left\| \bar{f}_{\omega} \right\|_{\RKHS}^2 .\]
\item For any ordering $\omega_1,\omega_2,\ldots$ of the Koopman eigenfrequencies, one has for every $l\in\num$,
\[\lim_{j\to\infty} \lim_{N\to\infty} \norm_{N,l}( f_{\omega_j}) = 0.\]
\end{enumerate}
\end{theorem}

The quantities $\norm_{N,l}$, which are explicitly defined in \eqref{eqn:trunc}, Section \ref{sect:proof_D}, measure the squared RKHS norm of $f_\omega| X_N$, projected onto an $l$-dimensional subspace of $\RKHS_N \subset \RKHS$ containing functions of minimal ``roughness'', as measured by the ratio $r_N $ from~\eqref{eqRKHSRat}. Due to this projection onto $\RKHS$ subspaces of fixed finite dimension $l$, Theorem~\ref{thm:D} provides a criterion for identifying Koopman eigenfunctions with RKHS representatives that differs from Theorem~\ref{thm:A}. As with Theorem \ref{thm:A}(iii), the result only determines eigenfrequencies up to a unique, though unspecified, translate of a Fourier frequency $\omega$ of the form $\omega + 2 \pi q / \Delta t $, $q \in \integer$. Note that if the system has at least two rationally independent eigenfrequencies, then every eigenfrequency has a translate $\omega$ in the Nyquist interval $[-\pi/\Delta t, \pi/\Delta t)$, and the set of such $\omega$ will be dense. The truncated RKHS norm provides a means of separating and ordering such frequencies. We shall describe a numerical procedure based on Theorems~\ref{thm:A} and~\ref{thm:D} (Algorithm \ref{alg:E}, Section~\ref{sect:numerics}), which employs both the value of $\norm_{N,l}$, as well as its growth with respect to $l$, to determine two criteria to identify these frequencies. See Figures \ref{fig:2D_oscill}--\ref{fig:l63_skew_additive} for an illustration of this approach applied to low-dimensional dynamical systems.

\blue{We also note that the integers $l_m$ in Theorem~\ref{thm:D}~(ii) are the cumulative dimension of the first $m$ eigenspaces of the kernel integral operator associated with the kernel $k$. This construction will be made clearer in Proposition~\ref{prop:correl_A} below. The reason behind using the sequence $l_m$ and not arbitrary $l$ is that $\norm_{N,l_m}(f_{\omega})$ becomes the data-driven approximation of the norm of $f_\omega$ projected to a fixed subspace, and thus will have a limit. }

\paragraph{Relation to DFT-based approaches} Estimation techniques for Koopman eigenvalues based on harmonic averaging/DFT \cite{Mezic05} make use of a related result to Theorem~\ref{thm:D}(i), which states that a $\mathbb{C}$-valued observation map $F $ has nonzero $L^2(\mu)$ projection onto a Koopman eigenspace at eigenfrequency $\omega + 2 \pi q / \Delta t$, $ q \in \mathbb{Z}$, iff the Fourier coefficient $\mathcal{F}_{\omega,N} F$ from~\eqref{eqn:FFT} converges, in $L^2(\mu)$ sense, to a nonzero value as $N \to \infty$ (see also Lemma~\ref{lem:erg_Fouri} ahead). This provides a sufficient condition for detecting Koopman eigenfrequencies (modulo Nyquist issues), but note that, unlike Theorem~\ref{thm:D}(i), the converse is not true. That is, if $\mathcal{F}_{\omega,N} F$ vanishes as $N \to \infty$, it could still be the case that $\omega + 2 \pi q / \Delta t $ is an eigenfrequency, for $F$ could be $L^2(\mu)$-orthogonal to the corresponding Koopman eigenspace. On the other hand, as will become clear below, the RKHS approach tests for eigenfrequencies using a complete orthonormal basis of $L^2(\mu)$ consisting of kernel eigenfunctions with representatives in $\mathcal{H}$, so that $w_{N,l}(f_\omega)$ is non-vanishing as $N \to \infty $ for sufficiently large $l$ iff $ \omega + 2 \pi q /\Delta t $ is an eigenfrequency. In fact, the result is applicable even for observation maps that do not take values in a linear space, such as manifold-valued maps. A practical scenario with manifold-valued observation maps is the analysis of directional data in geophysical applications (e.g., magnetic directional fields), taking values in the two-sphere.

\paragraph{Uniformity of convergence} Our final result is about systems which have no nonzero eigenfrequencies, and whose continuous spectrum is absolutely continuous with respect to Lebesgue measure. By Theorem~\ref{thm:D}(i), in such systems $\norm_{N,l}(f_{\omega})$ converges to zero for each $\omega\in \real\setminus\{0\}$. \blue{Let $\mathbb{E}_{\mu}$ denote the expectation with respect to $\mu$ of any quantity which is a function of the initial point $x_0$.} Recalling that the spaces $\RKHS_N$ and quantities $\norm_{N,l}(f_{\omega})$ depend on $x_0$, the following result establishes that the convergence is uniform over non-zero $\omega\in\real$, in an $L^1$ sense. 

\begin{theorem} \label{thm:F}
Let Assumptions \ref{assump:A1} and \ref{assump:A2} hold, and the quantities $\norm_N$ and $\norm_{N,l}$ be as in Theorem~\ref{thm:D}. Further, assume that the dynamics has no nonzero eigenfrequencies, and the continuous spectrum of the Koopman group $\{U^t\}$ is absolutely continuous with respect to the Lebesgue measure on the unit circle. Let $l\in\num$ be fixed and for every $N\in\num$, $\omega\in\real$, $\beta_{N,\omega}$ be the function which maps $x_0\in X$ into the quantity $\norm_{N,l}(f_\omega)$ calculated on the trajectory starting at $x_0$. Then, 
\[ \blue{ \lim_{N\to\infty} \sup_{\omega\in \real \setminus \{0\}} \mathbb{E}_{\mu} \beta_{N,\omega} = \lim_{N\to\infty} \sup_{\omega\in \real \setminus \{0\}} \mathbb{E}_{\mu} \norm_{N,l}(f_\omega) = 0 . }\]
\end{theorem}

\begin{rk*} One of the consequences of having an absolutely continuous spectrum is that the discrete component $\Disc$ from \eqref{eqn:L2_decomp} contains only constant functions. These systems are always weak-mixing \citep[][Chapter ``Mixing'']{Halmos1956}, but the converse is not true. For example, there are weak-mixing systems whose spectral measure, besides having an absolutely continuous component, also has a singular continuous component \citep[][p.~118]{glasner2015}. 
\end{rk*}

\paragraph{Discontinuity of the limit in Theorem~\ref{thm:D}} \blue{Note that even in the case of an absolutely continuous spectrum, $\omega=0$ is still an eigenfrequency, and by Theorem~\ref{thm:D}(ii), for $\mu$-a.e.\ $x_0 \in X$ and every $l \in \mathbb N$, $\lim_{N\to\infty} w_{N,l}(f_0) $ is equal to 1. On the other hand, for the same $x_0$ and $l$ and $\omega$ arbitrarily close to $0$, $\lim_{N\to\infty} w_{N,l}(f_\omega)$ is equal to 0. Thus, although $w_{N,l}(f_\omega)$ is a continuous function of $\omega$ for every $ N \in \mathbb N$, as $N\to \infty$, the limit has discontinuities at the eigenfrequencies of the dynamics. The occurrence of these discontinuities is due to the fact that the convergence as $N\to\infty$ is not uniform with respect to $\omega$. In fact, this singular behavior of the limit is a common characteristic of spectral methods in signal processing, such as DFT.}

Note that in practice one always scans for eigenvalues over a countable subset $\Omega'\subset \real$; for example, in the case of DFT, 
\begin{displaymath}
\Omega' = \{ j\pi/(N \, \Delta t): N \in \{ 1, 3, \ldots\},\; j \in \{ - (N-1)/2, \ldots, (N-1)/2 \} \}, 
\end{displaymath}
where we have assumed that the number of samples $N$ is odd for simplicity. We will use the same set of trial frequencies in the numerical implementation of our techniques, presented in Section~\ref{sect:numerics}.

\paragraph{Outline of the paper} We first prove Theorem \ref{thm:A} in Section \ref{sect:proof:A} by invoking Theorem \ref{thm:B}. In Section \ref{sect:theory_rkhs}, we review some important concepts from RKHS theory. Next, we prove Theorem \ref{thm:B} in Section \ref{sect:proof_B}, Theorem \ref{thm:D} in Section \ref{sect:proof_D}, and Theorem \ref{thm:F} in Section \ref{sect:proof_F}. In Section \ref{sect:numerics}, we discuss the numerical realization of our methods. In Section \ref{sect:examples}, the methods are applied to various systems with different types of spectrum, and compared with regular Fourier analysis of signals. 

\section{Proof of Theorem \ref{thm:A}} \label{sect:proof:A}

The proof will make use of the following lemma, which establishes three equivalent conditions for the ergodicity of the discrete-time map $\Phi^{\Delta t}$ stated in Assumption~\ref{assump:A1}.

\begin{lemma}\label{lem:ergdc_samplng}
Let $\Phi^t:M\to M$ be a continuous flow on a topological space $M$, and $\mu$ be an ergodic, invariant, Borel probability measure with compact support. Let $\Delta t>0$ be a sampling interval, resulting in the discrete-time map $\Phi^{\Delta t} : M\to M$. Then, the following are equivalent.
\begin{enumerate}[(i)]
\item $\mu$ is an ergodic measure under $\Phi^{\Delta t}$.
\item For every non-zero eigenfrequency $\omega \in \real$ of the Koopman group $ \{ U^t \} $ associated with $ \Phi^t $, $\omega\,\Delta t / 2\pi$ is not a nonzero integer.
\setcounter{enum_sav}{\value{enumi}}
\end{enumerate}
Furthermore, if either of the above conditions hold, then the following is true:
\begin{enumerate}[(i)]
\setcounter{enumi}{\value{enum_sav}}
\item For $\mu$-a.e.\ $x \in M$, the orbit of $x$ under $\Phi^{\Delta t}$ is dense in $\supp(\mu)$.
\end{enumerate}
\end{lemma}
\begin{proof} We make use of the spectral theory of strongly-continuous unitary evolution groups \cite{DGJ_compactV_2018,Nadkarni}, according to which there exists a \emph{projection-valued measure} $E$ mapping Borel sets on the real line to orthogonal projection operators on $L^2(\mu)$, such that the generator $V$ and the Koopman operator $U^t$ can be expressed as the operator-valued integrals
\[ V = \int_{\real} i\omega\, dE(\omega), \quad U^t = e^{t V} = \int_{\real} e^{i\omega t}\, dE(\omega), \quad \forall t\in\real. \]
The point spectra of $V$ and $U^t$ correspond to the atomic components of the measure $E$. That is, $E(\{\omega\})$ is nonzero iff $\omega$ is an eigenfrequency, and the fact that $U^{\Delta t} = e^{\Delta t\, V} $ implies that $e^{i\omega\,\Delta t}$ is an eigenvalue of $U^{\Delta t}$. The corresponding eigenspaces coincide, and are equal to the range of the projection map $E(\{\omega\})$. In particular, the space of invariant functions of $U^{\Delta t}$ is the same the union of the eigenspaces of $V$ that correspond to eigenfrequencies $\omega$ for which $e^{i \omega\,\Delta t} = 1$, i.e., $\omega\,\Delta t / 2\pi$ is an integer. Thus, we have
\[ \left\{ \text{Space of $U^{\Delta t}$-invariant functions} \right\} = \bigcup \left\{ \ran E(\{\omega\}) : \text{$\omega$ is an eigenfrequency and $ \omega\,\Delta t / 2\pi\in\integer$} \right\},\] 
where the subspaces in the union in the right-hand side (RHS) are one-dimensional and orthogonal by ergodicity of $ \Phi^t $. Note now that $ \mu $ is an ergodic invariant measure under $ \Phi^{\Delta t}$ (i.e., (i) holds) iff the left-hand side (LHS) above is a one-dimensional space consisting of only constant functions, which holds iff the RHS consists of only constant functions. Indeed, by the invariance and ergodicity of $\mu$ under the flow $\Phi^t$, $ \ran( E\{ \omega \} ) $ contains only the constant functions iff $ \omega = 0 $, which implies that the RHS consists of only constant functions iff $ \omega \, \Delta t / 2\pi $ is not a nonzero integer (i.e., (ii) holds). We therefore conclude that (i) and (ii) are equivalent. To show that (i) implies (iii), recall that if $\mu$ is ergodic for $\Phi^{\Delta t}$, then for every $f\in C^0(M)$, the equality
\[ \lim_{N\to\infty} \frac{1}{N}\sum_{n=0}^{N-1} f(\Phi^{n\Delta t}(x)) = \int_M f d\mu , \]
holds pointwise for $\mu$-a.e.\ $x\in M $ and in $L^2(\mu)$ sense. Consider any such initial point $x$. If its orbit $\mathcal{O}(x)$ is not dense in $\supp(\mu)$, then there is an open set $S$ disjoint from the closure of $\mathcal{O}(x)$, such that $\mu(S)>0$. Let $f$ be any non-zero, non-negative function with support in $S$. Then clearly the above identity would not hold, as the LHS would be zero and the RHS strictly positive. This leads to a contradiction, proving that (i) implies (iii). 
\end{proof}

We will now prove Theorem~\ref{thm:A}, assuming that Theorem~\ref{thm:B} is true. We begin with Claim~(i). By Lemma~\ref{lem:ergdc_samplng}(ii), for $\mu$-a.e.\ $x_0\in X$, the orbit of $x_0$ is dense in $X$, satisfying the hypothesis of Theorem \ref{thm:B}. The ``if'' part of the claim is the contrapositive of Theorem~\ref{thm:B}(ii), without the claim on uniqueness of the RKHS extension of $f$. To verify the ``only if'' part, we have to show that if $\lim_{N\to\infty} \norm_N (f_\omega)=\infty$, then $f_\omega$ does not have an extension in $\RKHS$. This is equivalent to its contrapositive statement: If $f_\omega$ has an extension $\bar{f}_\omega \in \RKHS$, then $\norm_N (f_\omega)$ does not diverge. But in that case, $\norm_N (f_\omega) = \norm_N (\bar{f}_\omega)$ since $f_\omega$ and $ \bar f_\omega$ lie in the same $L^2(\mu_N)$ equivalence class for all $N \in \num$, and by Theorem~\ref{thm:B}(i),
\[ \lim_{N\to\infty} \norm_N (f_\omega) = \lim_{N\to\infty} \norm_N (\bar{f}_\omega) = \lim_{N\to\infty} \left\| T_N( \bar{f}_{\omega}|X_N) \right\|^2_{\RKHS} = \|\bar{f}_{\omega}\|^2_{\RKHS} < \infty, \]
proving Claim~(i). The identity above also proves Claim~(ii). 

To prove Claim (iii), we will first show that any extension $\bar{f}_{\omega} \in \RKHS$ of $f_{\omega}$ is an eigenfunction of $U^{\Delta t}$. Indeed, for any $x_n \in \mathcal{O} $, we have
\[U^{\Delta t} \bar{f}_{\omega}(x_{n}) = \bar{f}_{\omega}(x_{n+1}) = f_{\omega}(x_{n+1}) = e^{i\omega\Delta t (n+1)} = e^{i\omega\Delta t} e^{i\omega\Delta t n} = e^{i\omega\Delta t} f_{\omega}(x_{n}) = e^{i\omega\Delta t} \bar{f}_{\omega}(x_{n}).\]
Now since $\bar f_\omega$ lies in $\RKHS$, and $\RKHS$ has a continuous reproducing kernel, $\bar f_\omega$ is continuous. Moreover, since $\mathcal{O}$ lies dense in $X$, it follows that $U^{\Delta t } \bar f_\omega( x ) = e^{i\omega\, \Delta t} \bar f_\omega(x)$ for all $x \in X$. The rest of the claim follows from Lemma~\ref{lem:ergdc_samplng}. \qed

\section{Results from reproducing kernel Hilbert space theory}\label{sect:theory_rkhs}

In this section, we review a number of properties of RKHSs which will be employed in the proofs of Theorems~\ref{thm:B}--\ref{thm:F}. For a more detailed exposition of this material we refer the reader to \cite{parzen1970}, or one of the many other references on RKHS theory.

\paragraph{Convergence in RKHS and uniform norms} As stated in Section~\ref{sec:intro}, if the reproducing kernel $k : M \times M \to \cmplx $ of an RKHS $\RKHS$ has a continuous restriction on $X\times X$, then convergence in $ \RKHS(X) $ norm implies convergence in $C^0(X)$ norm, which in turn implies pointwise convergence on $X$. To verify this directly, note that for every $f,g\in\RKHS(X)$ and $x\in X$, it follows from the Cauchy-Schwartz inequality and the reproducing property of $\RKHS(X)$ that
\begin{displaymath}
\left|f(x)-g(x)\right| = \left| \langle k(x,\cdot), f-g \rangle_{\RKHS} \right| \leq \|f-g\|_{\RKHS}\|k(x,\cdot)\|_{\RKHS}
\end{displaymath}
and
\begin{displaymath} 
\|k(x,\cdot)\|_{\RKHS}^2 = \langle k(x,\cdot), k(x,\cdot) \rangle_{\RKHS} = k(x,x) \leq \lVert k \rVert_{C^0(X\times X)},
\end{displaymath}
respectively, leading to 
\begin{displaymath}
\|f-g\|_{C^0(X)}^2 = \max_{x\in M} \lvert f(x)-g(x) \rvert^2 \leq \|f-g\|_{\RKHS}^2 \max_{x\in X} \|k(x,\cdot)\|_{\RKHS}^2 \leq \lVert f- g \rVert_{\RKHS}^2 \lVert k \rVert_{C^0(X\times X)}.
\end{displaymath}

\paragraph{Kernel integral operators} Kernel integral operators are compact operators on $L^2$ spaces, which provide a convenient way of realizing the RKHSs associated with continuous kernels. Specifically, given a finite Borel measure $\nu$ with compact support $ \supp(\nu ) \subseteq M$, the kernel integral operator $K_\nu : L^2(\nu) \to C^0(M)$ associated with a continuous kernel $k$ on $\supp( \nu ) $ is defined by 
\[K_\nu := f \mapsto \int_M k(\cdot,x)f(x)\, d\nu(x).\]
Let $\RKHS_\nu = \RKHS(\supp(\nu))$ be the RKHS on $\supp(\nu)$ with kernel $k_\nu = k|\supp(\nu) \times \supp(\nu)$. It can be verified that $\ran K_\nu | \supp(\nu) $ is a dense subspace of $ \RKHS_\nu $, and $ K_\nu : L^2(\nu) \to \RKHS_\nu $ is compact. Moreover, the adjoint map $K_\nu^*:\RKHS_\nu\to L^2(\nu)$ coincides with the inclusion map from $C^0(\supp(\nu)) $ to $ L^2(\nu) $. As a result, for every $ f \in L^2(\nu ) $, $ g \in \RKHS_\nu $, and $ \nu $-a.e.\ $ x \in M $, 
\begin{equation}\label{eqn:SStar}
K_\nu ^* f(x) = f(x), \quad \langle g, K_\nu f \rangle_{\RKHS_\nu} = \langle g,f\rangle_{\nu}.
\end{equation}
Note that $\RKHS_\nu$ naturally embeds into $\RKHS$ through the linear isometry $ \sum_{n \in \num} a_n k_\nu( \cdot, x_n ) \in \RKHS_\nu \mapsto \sum_{n\in\num} a_n k( \cdot, x_n ) \in \RKHS $, which means that we can view $\RKHS_\nu$ as a closed subspace of $ \RKHS$. Clearly, $\RKHS_\nu = \RKHS(X)$ if $\supp(\nu) = X$. 

Consider now the operator $G_\nu:= K_\nu ^* K_\nu$ on $ L^2(\nu) $. This operator is a trace-class (and therefore Hilbert-Schmidt and compact), self-adjoint, positive-semidefinite operator, and there exists an orthonormal basis $\{ \phi_j \}_{j=0}^\infty$ of $L^2(\nu)$ consisting of its eigenvectors. By convention, we order the basis elements $\phi_j $ in order of decreasing corresponding eigenvalues, $\lambda_j\geq 0$, which converge monotonically to 0 as $ j \to \infty $ by compactness of $G_\nu$. We will say that $k$ is $L^2(\nu)$-strictly-positive if $G_\nu > 0$. Note that a strictly positive-definite kernel is $L^2(\nu)$-strictly-positive for any finite, compactly supported Borel measure $ \nu$. Let now $J_\nu = \{ j \in \num_0 : \lambda_j \neq 0 \} $ be the index set for the nonzero eigenvalues of $G_\nu$, and define the set $ \{ \psi_j \}_{j\in J_\nu} $, 
\begin{equation}\label{eqn:def:Psi}
\psi_j = \lambda_j^{-1/2} K_\nu \phi_j. 
\end{equation}
It follows from~\eqref{eqn:SStar} that the $ \psi_j $ form an orthonormal set on $\RKHS_\nu$. Moreover, because $ k $ is continuous, it follows from Mercer's theorem \cite{Mercer1909} that $ k(x,y) = \sum_{j \in J_\nu} \psi_j( x ) \psi_j(y ) $, uniformly on $ \supp(\nu) \times \supp(\nu) $, which implies in turn that $\{\psi_j\}_{j\in J}$ is an orthonormal basis of $\RKHS_\nu$. The range of $K_\nu$ can also be expressed as
\begin{equation}\label{eqRanS}
\ran K_\nu = \left \{ f = \sum_{j \in J} b_j\psi_j \in \RKHS_\nu :\; \sum_{j\in J} |b_j|^2 / \lambda_j < \infty \right\}.
\end{equation}
In particular, because $ \lambda_j \to 0 $ as $ j \to \infty $, this shows that the range of $K_\nu$ is always a proper dense subspace of $\RKHS_\nu$, as stated above, unless $ \RKHS_\nu $ is finite-dimensional. 

\paragraph{Nystr\"om extension} The map $T_N$ employed in \eqref{eqn:norm} can be constructed using a procedure called \emph{Nystr\"om extension}. For a general finite Borel measure $\nu$ with compact support in $M$, the Nystr\"om extension operator $T_{\nu} : D(T_{\nu}) \to \RKHS_\nu$ has domain
\begin{displaymath}
D(T_\nu) = \left \{ f = \sum_{j\in J_\nu} a_j\phi_j \in L^2(\nu) :\; \sum_{j\in J} \lvert a_j \rvert^2 / \lambda_j < \infty \right\}, 
\end{displaymath}
and its action on every such $f$ is given by
\begin{equation}\label{eqn:def:Nystrom}
T_{\nu}f = \sum_{j\in J_\nu} \lambda_j^{-1/2}a_j \psi_j. 
\end{equation}
A key property of this operator, which is a consequence of~\eqref{eqn:SStar}, is $ K^*_\nu T_\nu f = f $ for all $ f \in L^2(\nu ) $. Since $K^*_\nu$ is an inclusion map, this shows that $ T_\nu f = f $ $ \nu $-a.e., and thus that $T_\nu$ extends $L^2(\nu) $ equivalence classes in its domain to $\RKHS_\nu$ functions. Note that because $\RKHS_\nu$ is a subspace of $C^0(\supp(\nu))$, and distinct elements of $C^0(\supp(\nu))$ lie in distinct $L^2(\nu)$ equivalence classes, it follows that $ T_\nu f $ is the unique $\RKHS_\nu$ extension of $ f \in L^2(\nu)$. If, in addition, $k$ is $L^2(\nu)$-strictly-positive, then it follows from~\eqref{eqRanS} that $ D( T_\nu ) $ is dense in $L^2(\nu) $. Note that because $H_\nu$ embeds naturally and isometrically into $\RKHS$, $T_\nu$ can also be defined as an extension operator mapping into the latter space, so that $T_\nu f$ is an $ \RKHS $ (and thus continuous) extension of $ f \in D(T_\nu) $ defined on the whole of $M$. However, unless $ \supp(\nu) = M $, that extension may not be unique.

In what follows, the measure $\nu$ will be either the invariant ergodic measure $\mu$ with \text{support equal to $X$}, or a sampling measure $\mu_N$ with finite discrete support. We will use the abbreviated notations $K=K_\mu$ and $K_N = K_{\mu_N}$. Note, in particular, that the action $K_N : L^2(\mu_N) \to \mathcal{ H } $ on $L^2(\mu_N) $ equivalence classes corresponds to weighted averages of kernel sections, viz.
\[K_N f = \frac{1}{N}\sum_{n=0}^{N-1} k(\cdot,x_n)f(x_n).\] We will denote the eigenvalues and eigenfunctions of $ G_N := K_N^* K_N $ by $ \lambda_{N,j} $ and $ \phi_{N,j}$, respectively, and those of $G:=K^*K$ by $ \lambda_j $ and $ \phi_j $, respectively.
The Nystr\"om extension operator in~\eqref{eqn:norm} is given by $T_N = T_{\mu_N}$. That is, for every $f = \sum_{j\in J_N} a_j \phi_{N,j} \in D(T_N)$, where $ J_N := J_{\mu_N} $, we have: 
\begin{equation}\label{eqn:def:Nystrom_data}
T_{N}f = \sum_{j \in J_N} a_j \lambda_{N,j}^{-1/2} \psi_{N,j}, \quad \psi_{N,j}= \lambda_{N,j}^{-1/2} K_N \phi_{N,j}, \quad K_N^* \psi_{N,j} = \lambda_{N,j}^{1/2} \phi_{N,j}.
\end{equation}
Note that, in general, $ J_N \subseteq \{ 1, \ldots, N\} $, and equality holds if Assumption~\ref{assump:A2} is satisfied. In that case, $D(T_N) = L^2(\mu_N)$. In the case of $T := T_\mu$, the analogous expressions to~\eqref{eqn:def:Nystrom_data} read
\[ Tf = \sum_{j \in J} a_j \lambda_{j}^{-1/2} \psi_{j}, \quad \psi_{N,j}= \lambda_{N,j}^{-1/2} K_N \phi_{N,j}, \quad K_N^* \psi_{N,j} = \lambda_{N,j}^{1/2} \phi_{N,j}, \]
where $ f = \sum_{j \in J} a_j \phi_j \in D(T)$, and $ J := J_\mu$. Under Assumption~\ref{assump:A2}, $T$ is a densely-defined, unbounded operator.

\paragraph{Spectral convergence}The eigenfunctions $\phi_j \in L^2(\mu)$ and $ \phi_{N,j} \in L^2(\mu_N)$ corresponding to nonzero eigenvalues have extensions in $\RKHS$, and thus in $C^0(M)$, given by 
\begin{equation}\label{eqPhiC0}
\varphi_j = T \phi_j = \lambda_j^{-1/2}\psi_j, \quad \varphi_{N,j} = T_N \phi_{N,j} = \lambda_{N,j}^{-1/2}\psi_{N,j}, 
\end{equation}
respectively. Let now $W_j$ and $W_{N,j}$ be the (finite-dimensional) eigenspaces of $G$ and $ G_N$ corresponding to strictly positive eigenvalues $\lambda_j $ and $\lambda_{N,j}$, respectively. The following lemma, which is based on \citep[][Theorem~15]{VonLuxburgEtAl08} and \citep[][Corrolary~2]{DasGiannakis_delay_Koop}, establishes a convergence result for these functions in the large-data limit in $C^0(X)$ norm.

\begin{lemma}\label{lemSpecConv} 
Let Assumptions~\ref{assump:A1} and~\ref{assump:A2} hold. Then, there exists a set $X'\subseteq X$ with $\mu$-measure $1$, such that for every starting state $x_0\in X'$, the following hold: 
\begin{enumerate}[(i)]
\item For each nonzero eigenvalue $ \lambda_j $ of $ G $, $\lambda_{N,j}$ converges to $\lambda_j$ as $N\to \infty$. 
\item For every eigenfunction $ \phi_j \in W_j$, there exist eigenfunctions $ \phi_{N,j} \in W_{N,j} $ such that their continuous representatives $ \varphi_j $ and $ \varphi_{N,j} $, respectively, satisfy $ \lim_{N\to\infty} \lVert \varphi_{N,j} - \varphi_j \rVert_{C^0(X)} = 0$. 
\item For each $N\in\num$, the $L^2(\mu)$ projection of the continuous function $\varphi_{N,j}$ onto $W_j$ has a continuous representative $\tilde{\varphi}_{N,j}$, and $\lim_{N\to\infty} \left\| \varphi_{N,j} - \tilde{\varphi}_{N,j} \right\|_{C^0(X)} = 0$.
\end{enumerate}
\end{lemma}

\section{Proof of Theorem \ref{thm:B}}\label{sect:proof_B} 

We will need the following lemma for the proof, which shows that the data-driven finite-dimensional subspaces $\RKHS_N|X$ ``converge'' to the RKHS $\RKHS(X)$. Henceforth, for simplicity of notation we will abbreviate $\RKHS_N|X$ by $\RKHS_N$.

\begin{lemma}\label{lem:RKHS_limit} 
Let the assumptions of Theorem \ref{thm:B} hold. Then, for $\mu$-a.e.\ $x_0\in M$ and every $N\in\num$, the subspace $\RKHS_N \subset \RKHS(X)$ is $N$-dimensional, and $\RKHS_1, \RKHS_2, \ldots $ forms a nested sequence of subspaces of $\RKHS(X)$. Moreover, $\cup_{N\in\num}\RKHS_N$ is dense in $\RKHS(X)$; that is, $\RKHS(X)$ is the closure of the span of $\{k(y,\cdot) : y\in\Orbit\}$.
\end{lemma}

\begin{proof} First, note that since $\RKHS_N = \spn \{ k(x_0,\cdot), \ldots, k(x_{N-1},\cdot) \} $, it is clear that $\RKHS_N\subseteq \RKHS_{N+1}$. Moreover since the points $x_0, \ldots, x_{N-1}$ are all distinct for $\mu$-a.e.\ $x_0 \in M$, by Assumption \ref{assump:A2}(ii), the $N$ kernel sections $k(x_0,\cdot), \ldots, k(x_{N-1},\cdot)$ are linearly independent. Therefore, for all such $x_0$, $\RKHS_N$ is $N$-dimensional, and $\RKHS_N \subset \RKHS_{N+1}$.

Next, consider the closed subspace $W = \overline{\cup_{N\in\num}\RKHS_N} = \overline{\spn\{k(y,\cdot) : y\in\Orbit\}}$, where closure is taken with respect to $\RKHS$ norm. It has to be shown that $W=\RKHS(X)$, or equivalently that $W^\bot=\{0\}$. But $f\in W^\bot$ iff for every $y\in\Orbit$, $\langle k(y,\cdot),f\rangle_\RKHS=0$. Since $k | X\times X$ is the reproducing kernel of $\RKHS(X)$, this is equivalent to saying that $f|\Orbit\equiv 0$. Now, by Lemma~\ref{lem:ergdc_samplng}(iii), for $\mu$-a.e.\ $x_0\in M$, the orbit $\Orbit$ is dense in $X$, and thus, for every such $x_0$, $f$ vanishes on a dense subset of $X$. However, by Assumption~\ref{assump:A2}(i), we have $\RKHS(X) \subset C^0(X) $, and thus $f$ is continuous and equal to 0 on the entire space $X$. This shows that $W^\bot=\{0\}$, as claimed.
\end{proof}

Theorem \ref{thm:B} is now ready to be proved. 

\paragraph{Proof of Claim (i)} We will first express $f|X_N$ as the result of applying a combination of operators from $\RKHS(X)$ to $L^2(\mu_N)$. Let $P_N:C^0(X)\to L^2(\mu_N)$ be the restriction map, satisfying $P_N f(x_n) = f(x_n)$. Let also $P:C^0(X)\to L^2(\mu)$ be the canonical inclusion map. As stated above, because the kernel $k$ is continuous, $\RKHS(X)$ is a subspace of $C^0(X)$, and therefore there exist inclusion maps $\iota : \RKHS(X)\to C^0(X)$ and $\iota_N:\RKHS_N\to C^0(X)$. The commutative diagram below shows how $P_N$, $\iota$, $ \iota_N $, and $T_N$ are related: 
\begin{equation}\label{eqn:TPj_commute}
\begin{tikzcd}[column sep=large]
L^2(\mu_N) \arrow{r}{T_N} \arrow{d}{\Id} & \RKHS_N \arrow{d}{\iota_N} \arrow{r}{\subset} &\RKHS(X) \arrow{dl}{\iota} \\
L^2(\mu_N) & C^0(X) \arrow{l}{P_N} 
\end{tikzcd}.
\end{equation}

Since $\RKHS_N$ is a finite-dimensional and hence closed subspace of $\RKHS(X)$, there exists an orthogonal projection $\pi_N:\RKHS(X) \to\RKHS(X)$ with $ \ran \pi_N = \RKHS_N $. By Lemma~\ref{lem:RKHS_limit}, for $\mu$-a.e.\ $x_0 \in M$, $\RKHS_N$ is a sequence of nested subspaces whose union is dense in $\RKHS$, and therefore, for every such $x_0 $, 
\begin{equation}\label{eqn:iN_piN}
\lim_{N\to\infty}\| \pi_N f-f\|_{\RKHS} =0,\quad \forall f\in\RKHS(X).
\end{equation}
Next, observe that for every $f\in\RKHS(X)$, $f|X_N = P_N \iota f$, and therefore $T_N ( f | X_N ) = T_N P_N \iota f$. Thus, Claim~(i) will be proved if it can be shown that the map $\pi'_N =T_N P_N \iota :\RKHS(X) ~\to~\RKHS(X)$ is the same as the orthogonal projection $\pi_N$, for, in that case, $\| \pi_N f\|_{\RKHS}$ and thus $\| T_N f|X_N\|_{\RKHS}$ will converge to $\|f\|_{\RKHS}$ by~\eqref{eqn:iN_piN}. 

To prove that $ \pi'_N = \pi_N $, it will be first shown that $\pi'_N$ is an idempotent operator, i.e., $\pi'_N \pi'_N = \pi'_N$. Indeed, by \eqref{eqn:TPj_commute}, $P_N \iota T_N$ is the identity map on $L^2(\mu_N)$, and therefore 
\[\pi'_N \pi'_N = T_NP_N \iota T_N P_N \iota = T_N (P_N \iota T_N) P_N \iota = T_N P_N \iota = \pi'_N.\]
Second, it will be shown that the range of $\pi'_N-I$ is orthogonal to $\RKHS_N$. Again by \eqref{eqn:TPj_commute}, 
\[ (P_N \iota) \circ \pi'_N = P_N \iota T_N P_N \iota \equiv P_N \iota,\]
which implies that for every $ g \in \RKHS(X) $ and $ n \in \{ 0, \ldots, N - 1 \} $, $(\pi'_N g)(x_n) = g(x_n)$, and therefore 
\[\langle \pi'_N g-g, k(\cdot,x_n) \rangle_{\RKHS} = (\pi'_N g)(x_n) - g(x_n) = 0\]
Now, since $\RKHS_N$ is spanned by $ \{ k(x_0,\cdot), \ldots, k( x_{N-1}, \cdot ) \} $, $\pi'_N g-g$ is orthogonal to $\RKHS_N$, and therefore $\pi'_N$ is an orthogonal projection into $\RKHS_N$, as claimed. This completes the proof of Claim(i). 
\qed

\paragraph{Proof of Claim (ii)} Under the assumptions of the claim, $ \{ T_N( f | X_N ) : N\in\num \} $ is a bounded sequence in $\RKHS(X)$. Therefore, because every bounded sequence in a Hilbert space has a weakly convergent subsequence, $T_N( f | X_N ) $ has a weakly convergent subsequence. The following proposition completes the proof. \qed

\begin{proposition}\label{prop:D}
Under Assumption~\ref{assump:A2}, if $f:\Orbit \to\cmplx$ is such that the sequence of functions $h_N= T_N\left( f|X_N \right)$ has a weakly-convergent subsequence in $\RKHS(X)$, then $f$ has a unique extension to $\RKHS(X)$.
\end{proposition}

\begin{proof} Let $\left( h_{N_j} \right)_{j=0}^{\infty}$ be such a weakly-convergent subsequence, and $h$ its weak limit. Note that by definition of $ T_N $, for fixed $ x_n \in \mathcal{O} $ and every $N>n$, $h_N(x_n)$ is constant and equal to $f(x_n)$. Therefore, by definition of weak convergence, 
\[ h(x_n) = \langle k(x_{n},\cdot),h \rangle_{\RKHS} = \lim_{j\to\infty} \langle k(x_n,\cdot), h_{N_j} \rangle_{\RKHS} = \lim_{j\to\infty} h_{N_j}(x_n) = f(x_n), \quad \forall n \in \{ 0,1,2,\ldots \},\]
which shows that $h$ is an extension of $f$ to $\RKHS(X)$. The uniqueness of $ h $ follows from the fact that it is continuous (since $\RKHS(X) \subset C^0(X) $), and $ \mathcal{O} $ is dense.
\end{proof}

As a side note, we will mention the following corollary of Lemma \ref{lem:RKHS_limit}.

\begin{cor}\label{corRKHS}
For every $t \in\real$, the space $\RKHS\circ\Phi^t$ obtained by composing every element of $\RKHS$ by the flow $\Phi^t$, is an RKHS with reproducing kernel $ k^{(t)}:M\times M\to\real$, $ k^{(t)}(x,y) = k\left( \Phi^t(x), \Phi^t(y) \right) $. In particular, if the kernel $k$ is invariant under the flow, i.e., if for every $t\in\real$, $k \equiv k^{(t)}$, then $\RKHS\circ\Phi^t = \RKHS$.
\end{cor}

\begin{proof} Let $\RKHS^{(t)}$ be the RKHS with reproducing kernel $k^{(t)}$, and for every $N\in\num$, let $\RKHS_{N}^{(t)}$ be the subspace of $\RKHS^{(t)}$ generated by the trajectory $\{ \Phi^{-t}(x_n) : n \in \{ 0,\ldots,N-1\} \}$, analogously to $\RKHS_{N}$. Then, 
\[\RKHS_N \circ \Phi^{t} = \RKHS_{N}^{(t)},\]
which, in conjunction with Lemma~\ref{lem:RKHS_limit}, implies that $ \RKHS \circ \Phi^{t} = \RKHS^{(t)}$. Now assume that $k$ is flow-invariant. In that case, because $f\in\RKHS$ is a sum of the form $f=\sum_{i\in\num} a_i k(y_i,\cdot)$, for any $t\in\real$,
\[ \| f\|_{\RKHS}^2 = \sum_{i,j\in\num} a^*_i a_j k(y_i, y_j) = \sum_{i,j\in\num} a_i^* a_j k(\Phi^{-t}(y_i), \Phi^{-t}(y_j)) = \left\| \sum_{i\in\num} a_i k(\Phi^{-t}(y_i),\cdot) \right\|_{\RKHS}^2,\]
and we conclude that $\tilde{f} = \sum_{i\in\num} a_i k(\Phi^{-t}(y_i),\cdot) $ lies in $ \RKHS$. However, 
\[\tilde{f}(x) = \sum_{i\in\num} a_i k(\Phi^{-t}(y_i),x) = \sum_{i\in\num} a_i k(y_i, \Phi^{t}(x) ) = f\circ\Phi^t(x), \]
which shows that $f\circ\Phi^t = \tilde{f}$ also lies in $\RKHS$, i.e., $\mathcal{H}\circ\Phi^t \subseteq \mathcal{H}$. The fact that $\mathcal{H} \circ \Phi^t = \mathcal{H}$ follows by replacing $t$ by $-t$ in the last inclusion.
\end{proof}

By Corollary \ref{corRKHS}, one of the consequences of having a flow-invariant kernel is that one can define a group of Koopman operators $ U^t : \mathcal{H } \to \mathcal{ H } $, $ t\in \real $, acting on the corresponding RKHS, $ \RKHS$. Flow-invariant kernels are a more special class of kernels, as they incorporate information about the underlying dynamics. One way of realizing them in a data-driven environment is as sequences of kernels operating on delay-embedded data with an increasing number of delays \cite{DasGiannakis_delay_Koop}. It was shown in \cite{DasGiannakis_delay_Koop} that the kernel integral operators associated with such kernels have common eigenspaces with the unitary Koopman operators on $ L^2(\mu) $. Other studies on Koopman operators on RKHSs, such as \cite{Kawahara2016,KlusEtAl2017}, make the strong assumption that the RKHS is Koopman-invariant. We do not know of any example of such spaces, other than the case of flow-invariant kernels in Corollary~\ref{corRKHS}.

\section{Proof of Theorem \ref{thm:D}} \label{sect:proof_D}

We begin with a lemma which shows that the projection of an $L^2(\mu)$ vector onto a Koopman eigenspace is also the limit of an exponentially weighted Birkhoff average of the function. \blue{Recall that due to the ergodicity of $\mu$, for every eigenfrequency $\omega$, $e^{i \omega\,\Delta t}$ is an eigenvalue for $U^{\Delta t}$ with multiplicity $1$, and thus there exists a nonzero orthogonal projection $\pi_\omega$ onto this eigenspace. If $\omega$ is not an eigenfrequency, then $\pi_\omega$ will be defined as the zero operator.}

\begin{lemma}\label{lem:erg_Fouri}
\blue{ Let Assumption \ref{assump:A1} hold. Then,} the orthogonal projection $\pi_\omega$ is given by the $L^2(\mu)$ limit
\[ \lim_{N\to\infty}\frac{1}{N}\sum_{n=0}^{N-1} e^{-i\omega n\,\Delta t}U^{n\,\Delta t}f = \pi_\omega f, \quad \forall f\in L^2(\mu). \]
\end{lemma}

\begin{proof} Let $I_\omega$ be the subspace of fixed points of the unitary operator $\tilde{U}_\omega= e^{-i\omega \,\Delta t}U^{\Delta t}$ , and $ \proj_{I_\omega} : L^2(\mu) \to L^2(\mu) $ the corresponding orthogonal projection operator. This subspace is \{0\} if $\omega$ is not a Koopman eigenfrequency; otherwise, it is the eigenspace of $U^{\Delta t}$ corresponding to eigenvalue $e^{i\omega\,\Delta t}$, so that $ \proj_{I_\omega} = \pi_\omega $. By the von Neumann mean ergodic theorem \citep[e.g.,][]{Halmos1956}, $\sum_{n=0}^{N-1} \tilde{U}_\omega^n/N $ converges pointwise to $ \proj_{I_\omega} $, and therefore, for any $ f \in L^2(\mu) $,
\[ \lim_{N\to\infty}\frac{1}{N}\sum_{n=0}^{N-1} e^{-i\omega n\,\Delta t}U^{n\,\Delta t}f = \lim_{N\to\infty}\frac{1}{N}\sum_{n=0}^{N-1} \tilde{U}_{\omega}^{n}f = \mbox{proj}_{I_\omega}f = \pi_{\omega} f. \qedhere \]
\end{proof}

We will now proceed to define the quantities $\norm_{N,l}$ introduced in Theorem \ref{thm:D} using Fourier-like averages. Let $\langle \cdot, \cdot \rangle_{\mu_N} $ denote the inner product of $L^2(\mu_N)$. Under the assumptions of the theorem, $ \{ \phi_{N,0}, \ldots, \phi_{N,N-1} \} $ is an orthonormal basis of $L^2(\mu_N)$ consisting of eigenfunctions of $G_N$ with nonzero corresponding eigenvalues, $\lambda_{N,j}$. Therefore, the Fourier function $f_\omega | X_N $ can be expressed as 
\begin{equation} \label{eqFExpansion}
f_{\omega}|X_N = \sum_{j=0}^{N-1}\langle \phi_{N,j}, f_\omega|X_N \rangle_{\mu_N} \phi_{N,j}.
\end{equation}
Thus, by definition of the Nystr\"om extension \eqref{eqn:def:Nystrom_data}, 
\begin{equation} \label{eqNystromFourier}
T_N ( f_{\omega}|X_N ) = \sum_{j=0}^{N-1} \langle f_{\omega}|X_N, \phi_{N,j} \rangle_{\mu_N} \psi_{N,j} / \lambda_{N,j}^{1/2} 
\end{equation}
and
\begin{equation}\label{eqWN}
\norm_N(f_{\omega}) = \left\| T_N ( f_{\omega}|X_N ) \right\|_{\RKHS}^2 = \sum_{j=0}^{N-1} \left|\langle f_{\omega}|X_N, \phi_{N,j} \rangle_{\mu_N} \right|^2 / \lambda_{N,j}.
\end{equation}
It follows from the above that $\norm_N(f_{\omega})$ can be approximated by the sequence of spectrally truncated norms 
\begin{equation}\label{eqn:trunc}
\norm_{N,l}(f_{\omega}) := \sum_{j=0}^{l-1} \left|\langle f_{\omega}|X_N, \phi_{N,j} \rangle_{\mu_N} \right|^2 / \lambda_{N,j}, \quad l \in \{ 1, \ldots, N \}, 
\end{equation}
where the quantity $l$ plays the role of a \emph{spectral resolution parameter}. 
Note that $\norm_{N,N}(f_\omega) = \norm_N(f_\omega) $. 

According to \eqref{eqWN}, $w_N(f_\omega)$ depends on all of the eigenpairs $(\lambda_{N,j},\phi_{N,j})$ of $G_N$. However, for a given $ N $, as $j $ increases the eigenvalues generally become increasingly sensitive to the particular trajectory $X_N$; that is, $ w_N(f_\omega) $ has high sensitivity to sampling errors. On the other hand, $\norm_{N,l}(f_\omega)$ depends on a fixed number $l$ of eigenvalues and eigenfunctions, which converge as $ N \to \infty $ uniformly with respect to $ j \in \{ 0, \ldots, l - 1 \} $ for $\mu$-a.e.\ starting state. This makes $\norm_{N,l}$ more useful from a practical standpoint than $w_N$.

Next, observe that the quantities $\langle \phi_{N,j}, f_\omega|X_N \rangle_{\mu_N}$ in~\eqref{eqFExpansion} are functions of the starting state $x_0 $ in the sampled orbit. The following proposition establishes the limit as $N \to \infty $ of these quantities, as $L^2(\mu)$ functions of $x_0 $.

\begin{proposition}\label{prop:correl_A}
Let $W$ be an eigenspace of the integral operator $G$, spanned by the $m$ eigenfunctions $\phi_k, \ldots, \phi_{k+m-1}$. \blue{Then, the following is true for $\mu$-a.e.\ initial point $x_0$}: Let $\omega\in \real$, and define $z_\omega$ to be a unit-norm eigenfunction if $\omega$ is an eigenfrequency, otherwise set $z_\omega =0$. Then, 
\[ \lim_{N\to\infty} \sum_{j=k}^{k+m-1} \left| \langle f_{\omega}|X_N, \phi_{N,j} \rangle_{\mu_N} \right|^2= \sum_{j=k}^{k+m-1} \left\| \pi_{\omega} \phi_{j} \right\|^2_{\blue{\mu}} = \left\| \proj_W z_\omega \right\|^2_{\blue{\mu}}. \]
Thus, if $\omega$ is not a Koopman eigenfrequency, then for each $j$, $\lim_{N\to\infty} \langle f_{\omega}|X_N, \phi_{N,j} \rangle_{\mu_N}$ is $\mu$-a.e.\ equal to zero. Otherwise, the leftmost sum converges to the squared norm of projection of the $z_\omega$ into the eigenspace $W$.
\end{proposition}

\begin{proof} For brevity, we will denote the indices $k,\ldots,k+m-1$ by $J$. For every $N\in\num$ note that both $\{ \varphi_{N,j} : j\in J \}$ and $\{ \varphi_{j} : j\in J \}$ (see \eqref{eqPhiC0}) are orthonormal sets of continuous functions. It then follows by Lemma~\ref{lemSpecConv}(iii), that there exists a unitary map $\unitary_N$ on $L^2(\mu)$ which is an identity on $W^\bot$, such that
\[\lim_{N\to\infty} \left\| \varphi_{N,j} - ( \unitary_N \varphi_j ) \right\|_{C^0(X)} = 0, \quad \forall j\in J .\]
Now, note that for each $j\in J$,
\[\begin{split}
\lim_{N\to\infty} \langle f_{\omega}|X_N, \phi_{N,j} \rangle_{\mu_N} &= \lim_{N\to\infty} \frac{1}{N}\sum_{n=0}^{N-1} e^{-i n\omega\,\Delta t} \phi_{N,j}(x_n) = \lim_{N\to\infty} \frac{1}{N}\sum_{n=0}^{N-1} e^{-i n\omega\,\Delta t} \varphi_{N,j}(x_n) \\
&= \lim_{N\to\infty} \frac{1}{N}\sum_{n=0}^{N-1} e^{-i n\omega\,\Delta t} ( \unitary_N \varphi_j )(x_n) + \lim_{N\to\infty} \frac{1}{N}\sum_{n=0}^{N-1} e^{-i n\omega\,\Delta t} \left( \varphi_{N,j} - ( \unitary_N \varphi_j ) \right) (x_n) \\
&= \lim_{N\to\infty} \frac{1}{N}\sum_{n=0}^{N-1} e^{-i n\omega\,\Delta t} ( \unitary_N \varphi_j )(x_n) . 
\end{split}\]
Lemma \ref{lem:erg_Fouri} states that the operator $\frac{1}{N}\sum_{n=0}^{N-1} e^{-i\omega n \Delta t}$ converges pointwise to the the projection $\pi_\omega$. If $f_N$ is a sequence of functions lying in a bounded, finite dimensional disk, then $\lim_{N\to \infty} \frac{1}{N}\sum_{n=0}^{N-1} e^{-i\omega n \Delta t} f_N(x_n) = \lim_{N\to \infty} \proj_\omega f_N$. In our case, for every $j\in J$ and every $N\in\num$, $\unitary_N \varphi_j$ lies in the unit disk of the finite dimensional space $W$. Thus we can write
\[ \lim_{N\to\infty} \sum_{j\in J} \left| \langle f_{\omega}|X_N, \phi_{N,j} \rangle_{\mu_N} \right|^2 = \lim_{N\to\infty} \sum_{j\in J} \left| \frac{1}{N}\sum_{n=0}^{N-1} e^{-i n\omega\,\Delta t} ( \unitary_N \varphi_j ) (x_n) \right|^2 = \lim_{N\to\infty} \sum_{j\in J} \left| \pi_\omega (\unitary_N \varphi_j)(x_0) \right|^2 \]
The next important realization is the following equality.
\begin{equation}
\left| \pi_\omega(\phi)(x_0) \right| = \left\| \pi_\omega(\phi) \right\|, \quad \forall \phi\in L^2(\mu), \quad \forall \omega\in\real, \quad \mu \mbox{-a.e. } x_0.
\end{equation}
Therefore, for $\mu$-a.e. $x_0$,
\[ \lim_{N\to\infty} \sum_{j\in J} \left| \langle f_{\omega}|X_N, \phi_{N,j} \rangle_{\mu_N} \right|^2 = \lim_{N\to\infty} \sum_{j\in J} \left\| \pi_\omega (\unitary_N \varphi_j)\right\|^2 . \]
But now note that for each $N$, $\{ \unitary_N \varphi_{j} : j\in J \}$ is still an orthonormal basis of $W$, and thus
\[ \sum_{j\in J} \left\| \pi_\omega (\unitary_N \varphi_j) \right\|^2 = \sum_{j\in J} \left| \pi_\omega (\varphi_j)(x_0) \right|^2 = \left\| \proj_W z_\omega \right\|^2 . \]
\blue{Finally, the fact that the full-measure set for $x_0$ can be chosen independently of $\omega$ follows from the Wiener-Wintner theorem \cite{Wiener_Wintner_1941}.} This completes the proof of the proposition.
\end{proof}


Using \eqref{eqn:def:Nystrom}, and introducing the orthogonal projection $ \proj_{N,l} : L^2(\mu_N) \to L^2(\mu_N)$ mapping into $\spn \left\{ \phi_{N,0}, \ldots, \phi_{N,l-1} \right\}$, we define the spectrally truncated Nystr\"om operator $T_{N,l} : L^2(\mu_N) \to L^2(\mu_N)$ as $T_{N,l} = T_N \proj_{N,l}$, i.e.,
\begin{equation}\label{eqn:def:TNl}
T_{N,l} \left(\sum_{j=0}^{N-1} a_j \phi_{N,j} \right) := \sum_{j=0}^{l-1} \lambda_{N,j}^{-1/2} a_j \psi_{N,j} = \sum_{j=0}^{l-1} a_j \varphi_{N,j}.
\end{equation}
The claims of Theorem \ref{thm:D} can now be proved. 

\paragraph{Proof of Claims (i) and (ii)} Since the RHS of \eqref{eqn:trunc} has a fixed, finite number of summands, we have 
\[ \lim_{N\to\infty} \norm_{N,l}(f_{\omega}) = \lim_{N\to\infty} \sum_{j=0}^{l-1} \left|\langle f_{\omega}|X_N, \phi_{N,j} \rangle_{\mu_N} \right|^2 / \lambda_{N,j} = \sum_{j=0}^{l-1} \lim_{N\to\infty} \left|\langle f_{\omega}|X_N, \phi_{N,j} \rangle_{\mu_N} \right|^2 / \lambda_{N,j}.\]
If $\omega + 2\pi q/\Delta t$ is not an eigenfrequency for any integer $q$, then by Proposition \ref{prop:correl_A}, each of the $l$ limits is equal to $0$ $\mu$-a.s., proving Claim~(i). \blue{ For Claim~(ii), choose $l_m$ to be the cumulative dimension of the first $m$ eigenspaces of the kernel integral operator $G$, as constructed in Proposition~\ref{prop:correl_A}. If $\omega + 2\pi q/\Delta t$ is an eigenfrequency for some $q\in\integer$, then the projection $\pi_\omega$ is nonzero, and because the $\phi_j$ form an orthonormal basis of $L^2(\mu)$, $\pi_\omega\phi_j$ is nonzero for some $j = L$. Choose $m$ large-enough so that $l_m>L$. This choice of $m$ and $\epsilon = \lVert \pi_\omega \phi_{L-1} \rVert_\mu$ } suffice for the first part of Claim~(ii). To prove the second part, let $z_{\omega}$ be a unit $L^2(\mu)$-norm Koopman eigenfunction corresponding to eigenfrequency $\omega + 2\pi q /\Delta t$. Then, for $\mu$-a.e.\ $x_0 \in M$, 
\[ T_{N,l} f_\omega = \sum_{j=0}^{l-1} \langle \phi_{N,j}, f_\omega \rangle_{\mu_N} \varphi_{N,j} , \quad K^* T_{N,l} f_\omega = \sum_{j=0}^{l-1} \langle \phi_{N,j}, f_\omega \rangle_{\mu_N} K^* \varphi_{N,j},\]
and it follows by Lemma~\ref{lemSpecConv}(ii) that
\[\begin{split}
\lim_{N\to \infty} K^* T_{N,l} f_\omega & = \lim_{N\to \infty} \sum_{j=0}^{l-1} \langle \phi_{N,j}, f_\omega \rangle_{\mu_N} K^* \varphi_{N,j} = \sum_{j=0}^{l-1} \lim_{N\to \infty} \langle \phi_{N,j}, f_\omega \rangle_{\mu_N} K^* \varphi_{N,j} \\ 
& = \sum_{j=0}^{l-1} \left( \pi_\omega \phi_j \right)(x_0) K^* \varphi_{j} = \sum_{j=0}^{l-1} \left( \pi_\omega \phi_j \right)(x_0) \phi_{j} = \sum_{j=0}^{l-1} \langle z_\omega, \phi_j \rangle_{\mu} z_\omega(x_0) \phi_j \\
& = z_\omega(x_0) \sum_{j=0}^{l-1} \langle z_\omega, \phi_j \rangle_{\mu} \phi_j = z_\omega(x_0) \proj_l \left( z_\omega \right) = \proj_l \left( z_\omega(x_0) z_\omega \right),
\end{split}\]
where $\proj_l$ is the orthogonal projection onto $\spn \left\{ \phi_{0}, \ldots, \phi_{l-1} \right\}$. Now for $\mu$-a.e.\ $x_0 \in M$, $|z_\omega(x_0)|=1$, so $z = z_\omega(x_0) z_\omega$ is again a unit-norm Koopman eigenfunction. The second part of Claim~(ii) follows by taking the limit $l\to\infty$, since $ \proj_l $ converges pointwise to the identity in that limit. 

\paragraph{Proof of Claim (iii)} If $\omega\in \Omega$, then $\pi_\omega = \langle z_\omega, \cdot \rangle_\mu z_\omega$, where $z_\omega$ is a Koopman eigenfunction at eigenfrequency $\omega$, lying in the domain of the Nystr\"om operator $T$. Then again by Proposition \ref{prop:correl_A}, $\lim_{N\to\infty} \norm_{N,l}(f_{\omega}) = \sum_{j=0}^{l-1} \left| \langle z_{\omega}, \phi_j \rangle_{\mu} \right|^2/\lambda_{j}$, and because $ z_\omega \in D(T)$, $\lim_{l\to\infty} \lim_{N\to\infty} w_{N,l}(f_\omega) = \lVert T z_\omega \rVert^2_{\mathcal{H}} < \infty $. By multiplying by a phase factor, we can make $Tz_\omega(x_0) $ equal to 1, and thus $ T z_\omega$ equal to the unique RKHS extension $\bar f_\omega$ from Theorem~\ref{thm:A}, proving Claim~(iii). 

\paragraph{Proof of Claim (iv)} For any ordering $\omega_m$ of the Koopman eigenfrequencies, the corresponding eigenfunctions $\left( z_{\omega_m} \right)_{m\in\num}$ form an orthonormal sequence in $L^2(\mu)$, and therefore converge weakly to $0$. By Proposition \ref{prop:correl_A}, for $\mu$-a.e.\ $x_0\in M$,
\[\lim_{m\to\infty} \lim_{N\to\infty} \norm_{N,l}( f_{\omega_m}) = \lim_{m\to\infty} \lim_{N\to\infty} \sum_{j=0}^{l-1} \left| \langle f_{\omega_m}|X_N, \phi_{N,j} \rangle_{\mu_N} \right|^2 / \lambda_{N,j} = \lim_{m\to\infty} \sum_{j=0}^{l-1} \left| \langle f_{\omega_m}, \phi_{j} \rangle_{\mu} \right|^2 / \lambda_{j}. \]
Since $l<\infty$, the limit can be brought inside the sum, and by the weak convergence of the $ z_{\omega_m} $ to zero,
\[\lim_{m\to\infty} \lim_{N\to\infty} \norm_{N,l}( f_{\omega_m}) = \sum_{j=0}^{l-1} \lim_{m\to\infty} \left| \langle z_{\omega_m}, \phi_{j} \rangle_{\mu} \right|^2 / \lambda_{j} = 0,\]
proving Claim (iv), and concluding the proof of Theorem \ref{thm:D}. \qed

\section{Proof of Theorem \ref{thm:F}} \label{sect:proof_F}

Using Lemma~\ref{lemSpecConv}(iii), we can rewrite \eqref{eqWN} as
\begin{equation}\label{eqn:fgaer}
\norm_{N,l}(f_\omega) = \sum_{j=0}^{l-1} \left| \langle f_{\omega}, \varphi_{N,j} \rangle_{\mu_N} \right|^2 \leq \sum_{j=0}^{l-1} \left| \langle f_{\omega}, \tilde{\varphi}_{N,j} \rangle_{\mu_N} \right|^2 + \sum_{j=0}^{l-1} \left| \langle f_{\omega}, \tilde{\varphi}_{N,j} - \varphi_{N,j} \rangle_{\mu_N} \right|^2 .
\end{equation}
It is important to keep in mind that for any $g \in C^0(M)$, the inner product $ \langle f_\omega, g \rangle_{\mu_N} $ is a continuous function of the starting point $x_0$ of the orbit $ \mathcal{O} $. In particular, the term on the LHS of \eqref{eqn:fgaer}, and the two terms on the RHS, are all functions of $x_0$. 

Before inspecting the dependence of $w_{N,l}(f_\omega)$ on $\omega$, as estimated by~\eqref{eqn:fgaer}, we will use Lemma~\ref{lemSpecConv} to establish a bound on the terms involving the difference $\tilde{\varphi}_{N,j} - \varphi_{N,j}$, {valid for every starting point $x_0 $ in the full-measure set $X'$ from Lemma~\ref{lemSpecConv}, viz.
\[\begin{split}
\left| \langle f_{\omega}, \tilde{\varphi}_{N,j} - \varphi_{N,j} \rangle_{\mu_N} \right| &= \left| \frac{1}{N}\sum_{n=0}^{N-1} f_\omega(x_n) \left[ \tilde{\varphi}_{N,j}(x_n) - \varphi_{N,j}(x_n) \right] \right| \leq \frac{1}{N}\sum_{n=0}^{N-1} \left| f_\omega(x_n) \right| \left| \tilde{\varphi}_{N,j}(x_n) - \varphi_{N,j}(x_n) \right| \\
&= \frac{1}{N}\sum_{n=0}^{N-1} \left| \tilde{\varphi}_{N,j}(x_n) - \varphi_{N,j}(x_n) \right| \leq \left\| \tilde{\varphi}_{N,j} - \varphi_{N,j} \right\|_{C^0(X)}.
\end{split}\]
Using this result, and taking the supremum over $\omega \in \real$ in the inequality in \eqref{eqn:fgaer}, we obtain
\[\sup_{\omega\in\real \setminus\{0\} } \norm_{N,l}(f_\omega) \leq \sum_{j=0}^{l-1} \sup_{\omega\in\real \setminus\{0\} } \left| \langle f_{\omega}, \tilde{\varphi}_{N,j} \rangle_{\mu_N} \right|^2 + \sum_{j=0}^{l-1} \left\| \tilde{\varphi}_{N,j} - \varphi_{N,j} \right\|_{C^0(X)}^2 . \]
Thus, taking the $N\to\infty$ limit of the LHS in the above equation, and distributing the limit over the finite number of terms on the RHS, leads to the bound 
\[ \lim_{N\to\infty} \sup_{\omega\in\real \setminus\{0\} } \norm_{N,l}(f_\omega) \leq \sum_{j=0}^{l-1} \lim_{N\to\infty} \sup_{\omega\in\real \setminus\{0\} } \left| \langle f_{\omega}, \tilde{\varphi}_{N,j} \rangle_{\mu_N} \right|^2 + \sum_{j=0}^{l-1} \lim_{N\to\infty} \left\| \tilde{\varphi}_{N,j} - \varphi_{N,j} \right\|_{C^0(X)}^2. \]
By Lemma~\ref{lemSpecConv}, each of the $l$ terms in the second sum vanish for all $ x_0 \in X'$, and hence the inequality 
\[ \lim_{N\to\infty} \sup_{\omega\in\real \setminus\{0\} } \norm_{N,l}(f_\omega) \leq \sum_{j=0}^{l-1} \lim_{N\to\infty} \sup_{\omega\in\real \setminus\{0\} } \left| \langle f_{\omega}, \tilde{\varphi}_{N,j} \rangle_{\mu_N} (x_0) \right|^2, \]
holds for every $x_0\in X'$. Here we have used the notation $\langle f_\omega, \tilde{\varphi}_{N,j} \rangle_{\mu_N}(x_0)$ to make the dependence of the inner product $\langle f_\omega, \tilde{\varphi}_{N,j} \rangle_{\mu_N}$ on $x_0$ explicit. \blue{We have thus shown that for fixed $l\in\num$,
\[ \lim_{N\to\infty} \sup_{\omega\in\real \setminus\{0\} } \beta_{N,\omega}(x_0) \leq \sum_{j=0}^{l-1} \lim_{N\to\infty} \sup_{\omega\in\real \setminus\{0\} } \left| \langle f_{\omega}, \tilde{\varphi}_{N,j} \rangle_{\mu_N} (x_0) \right|^2, \quad \forall x_0\in X'. \] }
Since this bound holds for $x_0$ lying in a full-measure set $X'$ which is independent of $\omega$ and $N$, taking expectation with respect to $x_0$ over this set gives
\blue{ \[ \lim_{N\to\infty} \sup_{\omega\in\real \setminus\{0\} } \mathbb{E}_{\mu} \beta_{N,\omega} \leq \sum_{j=0}^{l-1} \lim_{N\to\infty} \sup_{\omega\in\real \setminus\{0\} } \mathbb{E}_{\mu} \left| \langle f_{\omega}, \tilde{\varphi}_{N,j} \rangle_{\mu_N} \right|^2. \] }
Since there are only finitely many ($l$) summands, it is sufficient to show that for each $0\leq j<l$,
\blue{ \[ \lim_{N\to\infty} \sup_{\omega\in\real \setminus\{0\} } \mathbb{E}_{\mu} \left| \langle f_{\omega}, \tilde{\varphi}_{N,j} \rangle_{\mu_N} \right|^2 = 0. \] }
%
Note that by definition (see Lemma~\ref{lemSpecConv}), $\tilde{\varphi}_{N,j}$ lies in the unit ball of the finite-dimensional subspace $W_j \subset L^2(\mu)$, for every $N\in\num$. It is therefore sufficient to show that for every $\phi\in L^2(\mu)$,
\begin{equation}\label{eqn:Fourier_uniform_decay}
\lim_{N\to\infty}\sup_{\omega\in\real \setminus\{0\} } \int_M \lvert \langle f_{\omega}, \phi \rangle_{\mu_N}(x_0) \rvert^2 \, d\mu(x_0) =0.
\end{equation}

To that end, let $\Fourier_{\omega,N}$ be the Fourier averaging operator on the space of complex-valued functions on $M$, defined for every $\omega\in\real$ and $N\in\num$ as 
\[
\Fourier_{\omega,N} := \frac{1}{N} \sum_{n=0}^{N-1} e^{-i\omega n}U^{n\,\Delta t}.
\]
Since $\Fourier_{\omega,N}$ is given by a finite linear combination of powers of $U^{\Delta t}$, it acts as an operator on $C^0(M)$, $C^0(X)$, and extends to an operator on the $L^p(\mu)$ spaces with $p\geq 1$. In particular, note that 
\[ \langle f_\omega, \phi \rangle_{\mu_N}(x_0) = \Fourier_{\omega,N} \phi(x_0), \quad
\int_X \left| \langle f_\omega, \phi \rangle_{\mu_N}(x_0) \right|^2 d\mu(x_0) = \left\| \Fourier_{\omega,N}\phi \right\|_{\mu}^2, \]
Thus, to prove the theorem, it is sufficient to show that
\begin{equation}\label{eqn:FFT_norm}
\lim_{N\to\infty}\sup_{\omega\in\real \setminus\{0\} } \left\| \Fourier_{\omega,N} \phi \right\|_{\mu}^2 =0, \quad \forall \phi\in L^2(\mu).
\end{equation}

We will now verify~\eqref{eqn:FFT_norm} by employing the spectral representation of unitary operators. In particular, let the circle $S^1$ be parameterized by the interval $[0,2\pi)$, with the ends identified. Let also $H_{\phi}$ be the cyclic subspace generated by $\phi$, i.e., the $L^2(\mu)$ closure of $ \spn \{ U^{n\,\Delta t} \phi : n\in\integer\} $. By Herglotz's theorem, there exists a Borel probability measure $\mu_{\phi}$ on $S^1$, such that, for every $n\in\integer$,
\[c_n := \langle U^n\phi, \phi \rangle_{\mu} = \int_{0}^{2\pi} e^{-i n \theta}d\mu_{\phi}(\theta).\]
As a result, the linear map $\Psi_{\phi}:H_{\phi}\to L^2(S^1,\mu_{\phi})$, defined uniquely through the requirement that $(\Psi_{\phi} U^n \phi )(\theta) = e^{i n \theta}$, for all $ n \in \integer $, is an isometry, and we have
\[ \left( \Psi_{\phi} \Fourier_{\omega,N} \phi \right)(\theta) = \frac{1}{N}\sum_{n=0}^{N-1} e^{-i\omega n} ( \Psi_{\phi} U^n \phi)(\theta) = \frac{1}{N}\sum_{n=0}^{N-1} e^{i(\theta-\omega)n} = \frac{1}{N}\frac{e^{i(\theta-\omega)N}-1}{e^{i(\theta-\omega)} -1}.\]
The term in the right-hand side of the last inequality can be succinctly expressed in terms of a function $\Sinc_N:\real\to\real$, which is very similar to the $N$-th Fourier cosine coefficient of the Bartlett window used in signal processing \citep[][pp.~62--64]{Bracewell2000}, namely, 
\[
\Sinc_N(u) := \left|\frac{\sin(N u/2)}{N\sin(u/2)} \right| = \left| \frac{1}{N}\frac{e^{iuN}-1}{e^{iu} -1} \right|, \quad
\Sinc_N(\theta-\omega) = \left| \frac{1}{N}\frac{e^{i(\theta-\omega)N}-1}{e^{i(\theta-\omega)} -1} \right| = \left| \left( \Psi_{\phi} \Fourier_{\omega,N} \phi \right)(\theta) \right|.\\
\]
Using this function, we obtain
%
\begin{displaymath}
\left\| \Fourier_{\omega,N} \phi \right\|_{\mu}^2 = \int_{0}^{2\pi} \left| \left( \Psi_{\phi} \Fourier_{\omega,N} \phi \right)(\theta) \right|^2 \, d\mu_{\phi}(\theta) = \int_{0}^{2\pi} \Sinc_N^2(\theta-\omega) \, d\mu_{\phi}(\theta).
\end{displaymath}
Therefore, 
\eqref{eqn:FFT_norm} will be proved if it can be shown that
\begin{equation}\label{eqn:Fourier_uniform_decay_variance}
\lim_{N\to\infty}\sup_{\omega\in\real \setminus\{0\} } \int_{0}^{2\pi} \Sinc_N^2(\theta-\omega) \, d\mu_{\phi}(\theta) =0.
\end{equation}
Fixing an arbitrary $\epsilon>0$ and $\omega\neq 0$, it will be shown that for $N$ sufficiently large, $\left| \int_{0}^{2\pi} \Sinc_N^2(\theta-\omega) \, d\mu_{\phi}(\theta) \right| <2\epsilon$. For that, observe that by assumption, the measure $\mu_{\phi}$ is absolutely continuous with respect to the Lebesgue measure $\mu_\text{Leb}$ restricted to $S^1 \setminus\{0\} $, and hence it has a density $\rho_{\phi} = d\mu_\phi/ d\mu_\text{Leb} \in L^1(S^1 \setminus\{0\} ,\mu_\text{Leb})$. Since $\rho_\phi$ is integrable with respect to Lebesgue measure, there exists $\delta>0$ such that for any Borel set $E\subset \real$ with $\mu_{\text{Leb}}(E)<\delta$, $\int_E \rho_\phi \, d\mu_\text{Leb} <\epsilon$. Let now $I_\omega \subset S^1$ be an interval of length $\delta$, centered at $\omega \neq 0$, and assume that $\delta$ is small enough so that $0\notin I_\omega$. Then 
\[ \int_{0}^{2\pi} \Sinc_N^2(\theta-\omega) \, d \mu_{\phi}(\theta) = \int_{I_\omega} \Sinc_N^2(\theta-\omega) \rho_{\phi}(\theta) \, d\theta + \int_{S^1\setminus I_\omega} \Sinc_N^2(\theta-\omega) \, d \mu_{\phi}(\theta), \]
and because $0\leq \Sinc^2_N(u) \leq 1$ for every $u \in \real$, the first integral in the RHS is less than $\epsilon$. To bound the second integral, note that for any $\delta>0$ and as $ N\to \infty$,
\[\sup_{u\in\real:\;\lvert u \rvert>\delta } \Sinc_N(u) = O(1/N), \]
which implies that the second integral is $O(N^{-2})$, and thus less than $\epsilon$ for sufficiently large $N$. This proves \eqref{eqn:Fourier_uniform_decay_variance}, which proves \eqref{eqn:Fourier_uniform_decay}, and thus Theorem \ref{thm:F}. \qed

\section{Numerical computation of RKHS norms} \label{sect:numerics}

When numerically implementing the results of Theorems \ref{thm:A} and \ref{thm:D} to find Koopman eigenfrequencies and eigenfunctions, one has to deal with two limitations, namely: (i) instead of having access to Fourier functions $f_\omega$ on a full orbit $\mathcal{O}$, one only has access to a finite trajectory $X_N \subset \mathcal{O}$; and (ii) among the potentially countably-infinite set of Koopman eigenfrequencies, one can practically only identify a set of candidate eigenfrequencies. Regarding (i), note that the Nystr\"om extension $h_N = T_N(f_\omega|X_N)$ is a continuous function on $X$ (or $M$), and can be calculated even at points $x$ lying outside the orbit. By Theorem \ref{thm:A}, if $\omega$ has a translate by $2\pi q / \Delta t$, $ q \in \mathbb{Z} $, lying in the frequency set $\Omega$ from~\eqref{eqn:def:Omega}, then $h_N$ converges in $\RKHS(X)$ to a Koopman eigenfunction, and for large-enough, finite $N$, $ h_N $ is a good approximation of that eigenfunction. The limitation pointed out in (ii) is alleviated from the fact that, by virtue of a group structure that Koopman eigenfrequencies and eigenfunctions possess, which is described below, estimates of any finite collection of them can be used to generate arbitrarily many estimates.

\paragraph{Group structure of Koopman point spectra} By definition of the Koopman operator, the product of any two Koopman eigenfunctions $ z_{\omega_1}, z_{\omega_2} \in L^\infty(\mu) $ corresponding to eigenfrequencies $ \omega_1, \omega_2$, respectively, is also an $L^\infty(\mu) $ eigenfunction corresponding to the eigenfrequency $\omega_1 + \omega_2 $. A countable (finite or infinite) collection of rationally independent eigenfrequencies $\omega_1,\omega_2,\ldots$ is said to be a generating set, if for any eigenfrequency $\omega$, there exists $q \in \num$ and integer coefficients $c_1, \ldots, c_q $ such that $\omega = \sum_{j=1}^q c_j \omega_{k_j}$ for some $1 \leq k_1 < k_2 < \cdots < k_q $. Moreover, if $z_{k_1 }, \ldots, z_{k_q}$ are Koopman eigenfunctions of unit $L^2(\mu)$ norm corresponding to $ \omega_{k_1 }, \ldots, \omega_{k_q}$, respectively, then $z =\prod_{j=1}^{q} z_{k_j}^{c_j}$ is a unit-norm Koopman eigenfunction corresponding to $\omega$. The set of all such eigenfunctions forms an orthonormal basis of the point spectrum subspace $\mathcal{D} $ from~\eqref{eqn:L2_decomp}. In many dynamical systems, such as Kolmogorov-Arnold-Moser (KAM) tori, quasiperiodic systems, limit cycles, and periodically driven chaotic systems, there is a finite generating set $\omega_1, \ldots, \omega_{q}$ for some minimal number $\NGenFreq$. In such systems, a minimal generating set of eigenfrequencies is not unique, but is always of size $\NGenFreq$. In particular, any set of $\NGenFreq$ rationally independent eigenfrequencies constitutes a minimal generating set.

\paragraph{Experimental setup} In the data-driven modeling scenario we wish to consider here, the underlying dynamical system is unknown, or inaccessible to direct observation. Instead, we assume we have access to a finite, time-ordered dataset $F(x_0), \ldots, F(x_{N-1}) $, consisting of the values $F(x_n)$ of an observation function, $ F: M \to Y $, on an (unknown) finite trajectory $X_N $ as in Section~\ref{sec:intro}. The observation function will be assumed to have the following properties.

\begin{Assumption}\label{assump:F} 
$F : M \to Y $ is a map taking values in a metric space $Y$, such that $F| X $ is injective and continuous.
\end{Assumption}

Hereafter, we will refer to $Y$ as the data space. As stated in Section~\ref{sec:intro}, unlike conventional DFT-based spectral estimation approaches (e.g.,~\eqref{eqn:FFT}), the RKHS-based techniques proposed here do not require $Y$ to have the structure of a linear space. Note also that the injectivity requirement on $F$ can be generically relaxed through the use of delay-coordinate maps \cite{SauerEtAl91}; we will discuss this point further below. 

\paragraph{Choice of kernel} The experimental setup described above places a restriction on the type of kernel $k : M \times M \to \cmplx$ employed, as it must be computable from the values of $F$ alone. That is, $ k $ must have the structure of a ``pullback kernel'', $k(x,y) := \kappa(F(x), F(y))$, where $ \kappa: Y \times Y \to \cmplx $ is a kernel on $ Y $ designed according to the requirements of the application at hand. Here, we require that Assumption~\ref{assump:A2} be satisfied, which implies that $ F $ be injective and continuous on $X$ (i.e., Assumption~\ref{assump:F} is satisfied), and that the restriction of $ \kappa $ on $ F(X) \times F(X) \subseteq Y \times Y$ be continuous and strictly positive-definite.

As a guideline for choosing $ \kappa $ so as to satisfy the strict positive-definiteness condition, we note that the reproducing kernel $ k : M \times M \to \cmplx$ of an RKHS on $ M $ is strictly positive-definite if and only if the kernel sections $k( x_1, \cdot ), \ldots, k( x_n, \cdot ) $ are linearly independent for all $ x_1, \ldots, x_n $ in $ M $. When one does not have a priori knowledge of the image $ F( X ) $ of the support of the invariant measure in data space, it is generally preferable to define $ k $ through a $ \kappa $ which is strictly positive-definite on the whole of $ Y $. For example, in the case $Y=\real^m$, the radial Gaussian kernels,
\begin{equation}
\label{eqKGauss}
\kappa( y_1, y_2 ) = \exp \left( - \frac{ d^2( y_1, y_2 ) }{ \epsilon } \right),
\end{equation} 
where $d : Y \times Y \to \real$ is the Euclidean metric and $ \epsilon $ a positive bandwidth parameter, are strictly positive definite \cite{Micchelli86}. On the other hand, the covariance kernel,
\begin{equation}
\kappa(y_1, y_2) = \langle y_1, y_2 \rangle_Y = \left[ d^2( y_1, y_2 ) - d^2( y_1, -y_2 ) \right] / 4, 
\label{eqKCov}
\end{equation}
does not lead to a strictly positive-definite kernel $ k $ on $ X $, as in this case $ k( x, \cdot ) $ depends linearly on $ F( x ) $. We will return to a discussion of the behavior of our methods implemented with covariance kernels and their relationship to DFT approaches in Section~\ref{sect:examples}. Additional examples of commonly used kernels in machine learning and signal processing can be found in \cite{Genton01}.


\paragraph{Markov normalization} Another option in kernel selection, which we will adopt in Section~\ref{sect:examples}, is to start from a sign-definite (i.e., strictly positive-valued), strictly positive-definite, continuous kernel $ k : M \times M \to \real $, such as the Gaussian kernel in~\eqref{eqKGauss}, and manipulate it to obtain a normalized kernel $ \hat p_N : M \times M \to \real $, also strictly positive and strictly positive-definite. Following \citep[][Section~4.3]{DasGiannakis_delay_Koop}, we define
\begin{equation} \label{eqn:def:Markov}
\hat{p}_N(x,y) := \frac{k(x,y)} {\hat{\sigma}_N(x) \hat{\sigma}_N(y) }, \quad \hat{\sigma}_N := \sqrt{\sigma_N \rho_N}, \quad \rho_N := K_N (1_M), \quad \sigma_N := K_N(1_M/\rho_N). 
\end{equation}
By compactness of $X$, the functions $\rho_N$ and $\sigma_N$ are strictly positive, and their restrictions on $X$ are continuous. This makes $\hat p_N $ strictly positive, strictly positive-definite, and continuous on $X \times X$. In addition, one can verify that $\hat p_N $ is related to a non-symmetric kernel $ p_N : M \times M \to \real$ via the transformation
\begin{displaymath}
\hat p_N(x,y) = \sqrt{\frac{\sigma_N(x)}{\rho_N(x)}} p_N(x,y) \sqrt{\frac{\rho_N(y)}{\sigma_N(y)}}, \quad p_N( x, y ) = \frac{ k(x, y ) }{ \sigma_N( x ) \rho_N(y )}, 
\end{displaymath}
where $ p_N $ has the Markov property with respect to $ \mu_N$, i.e., $\sum_{n=0}^{N-1} p_N(x,x_n) / N = 1$, for all $x \in M$. As a result, the integral operator $ \hat P_N : L^2(\mu_N) \to L^2(\mu_N)$ associated with $ \hat p_N $ is related to an ergodic Markov operator on the same space by a similarity transformation; in particular, the eigenvalues $ \lambda_{N,j} $ of $P_N $ admit the ordering $ 1 = \lambda_{N,0} > \lambda_{N,1} \geq \lambda_{N,2} \geq \cdots \searrow 0$. Although $ \hat{p}_N$ and $p_N$ depend on $N$, it can be shown \citep[][Section~4.3]{DasGiannakis_delay_Koop} that as $N\to \infty$, the functions $ \rho_N $ and $ \sigma_N $ converge in $C^0(X)$ norm, $\mu$-almost surely, to the analogous functions computed with respect to the invariant measure $\mu$ instead of the sampling measure $\mu_N$, i.e., $ \rho = K 1_M $ and $ \sigma = K( 1_M / \rho ) $, respectively. As a result. $ \hat p_N $ and $ p_N $ converge in $C^0(X) $ norm to 
\begin{displaymath}
\hat{p}(x,y) := \frac{k(x,y)} {\hat{\sigma}(x) \hat{\sigma}(y) }, \quad p( x, y ) = \frac{ k(x, y ) }{ \sigma( x ) \rho(y )}, 
\end{displaymath}
respectively, where $ \hat \sigma = \sqrt{\sigma \rho} $, and $ p $ is Markov with respect to $\mu$.

Kernel normalizations such as~\eqref{eqn:def:Markov} are widely employed in manifold learning applications \cite{CoifmanLafon06,BerrySauer16b}, as they provide increased flexibility to approximate geometrical operators (e.g., heat operators) from data compared to unnormalized kernels, while increasing robustness to variations in the sampling density of the data. For our purposes, in addition to increasing robustness, kernel normalization has the benefit of normalizing the eigenvalue spectrum $ \lambda_{N,j} $ employed in the numerical calculation of RKHS norms---this will facilitate to some extent the tuning of threshold parameters in the procedure for selecting candidate Koopman eigenfrequencies described below.
\paragraph{Kernels from delay-coordinate mapped data} The condition that the observation map $ F|X : X \to Y $ is injective may not be satisfied in a number of real-world applications. In such cases, it is possible to utilize the dynamical flow to construct an empirically accessible observation map $ F_Q : M \to Y^Q $, $ Q \in \num $, using the method of delay-coordinate maps, viz.
\begin{equation}
\label{eqFQ}F_Q(x) = \left( F(x), F(\Phi^{-\Delta t}x), \ldots, F(\Phi^{-(Q-1)\Delta t}x) \right).
\end{equation}
It can be shown \cite{SauerEtAl91} that under mild genericity assumptions on $F$ and $ \Phi^{\Delta t} $, there exists a $ Q_0 \in \num $ such that for all $ Q \geq Q_0 $, $ F_Q | X $ is a homeomorphism onto its image. If these conditions are met, the methods described in this paper can be applied using the data obtained from $ F_Q $ instead of $ F $. 

\paragraph{The algorithm} The following algorithm describes a numerical procedure to compute approximate Koopman eigenfrequencies and their corresponding eigenfunctions in $\RKHS$ based on values of the corresponding kernel $k$ on the product trajectory $ X_N \times X_N $. Four parameters need to be supplied as inputs, namely integer spectral resolution parameters $ l_0$, $ l_1 $ with $0\leq l_0\leq l_1<N$, and positive threshold parameters $\delta_0$, $\delta_1$. The roles of $(l_0,\delta_0)$ and $(l_1,\delta_1)$ are to provide practically computable proxies to the asymptotic criteria established in Theorems~\ref{thm:D} and~\ref{thm:A}, respectively. In particular, if a trial Fourier frequency $ \omega \in \real $ has no $ 2\pi q/\Delta t $ translates, $q \in \integer$, lying in the set of Koopman eigenfrequencies, then by Theorem~\ref{thm:D}(i), for any $ l_0 \in \num $ and $ \delta_0 >0 $, $\norm_{N,l_0}(f_\omega)$ should be smaller than $ \delta_0 $ for large-enough $ N$. This suggests that, at fixed $N $, and for a given choice of $(l_0,\delta_0)$, the set of Fourier frequencies for which $w_{N,l_0}(f_\omega) < \delta_0$ should be good candidates for frequencies exhibiting the asymptotic behavior in Theorem~\ref{thm:D}(i), and thus can be rejected from the set of trial Koopman eigenfrequencies. In general, decreasing $ \delta_0 $ at fixed $N$ and $ l_0 $ decreases the likelihood of false positives (i.e., false detection of frequencies which are not Koopman eigenfrequencies), but also decreases the likelihood of true positives (i.e., the likelihood of failure to detect true Koopman eigenfrequencies increases). Similarly, at fixed $N$ and $\delta_0$, the likelihoods of false positives and true positives both decrease by increasing the parameter $l_0$, which controls the dimension of the RKHS subspace $ \spn\{ \psi_0, \ldots, \psi_{l_0-1} \}$ in which we search for eigenfunctions. Having pruned an initial set of candidate eigenfrequencies via the procedure just described, we proceed by subjecting the remaining frequencies to an additional test based on Theorem~\ref{thm:A}(i). According to that theorem, if $\omega $ does not have $ 2 \pi q / \Delta t $ translates in the frequency set $ \Omega $ from~\eqref{eqn:def:Omega}, then for any $ \delta_1 > 0 $, the ratio 
\begin{equation}\label{eqn:def:ratio}
\ratio(\omega;l_0,l_1,N) := \frac{\norm_{N,l_1}(f_\omega)}{\norm_{N,l_0}(f_\omega)} -1 ,
\end{equation}
should exceed $ \delta_1 $ for fixed $l_0$ and large-enough $l_1$ and $N$. When working with a dataset of fixed size $N$, this suggests rejecting all candidate frequencies for which $ r(\omega;l_0,l_1,N) > \delta_1 $ for some choice of $l_1 > l_0 $. In this test, decreasing $ \delta_1$ at fixed $(N, l_0, l_1 )$, or increasing $l_1 $ at fixed $ ( N, l_0, \delta_1 ) $, decreases the risk of false positives, while decreasing the likelihood of true positives.

Algorithm \ref{alg:E} below summarizes a procedure for identifying candidate Koopman eigenfrequencies and their corresponding eigenfunctions in $\RKHS $ by sequential application of the two rejection criteria described above on a candidate frequency set $\Omega_N$. In practical applications, that set must be necessarily finite, and due to the Nyquist frequency limitations discussed in Section~\ref{sec:intro}, it can be chosen as a subset of $[-\pi/\Delta t, \pi / \Delta t]$. The algorithm is meant to be used in conjunction with Proposition~\ref{prop:norm_algo_B} for evaluating $w_{N,l}(f_\omega) $ for $ \Omega_N \ni \omega$ set to the standard DFT frequency grid. By convention, it outputs candidate eigenvalue--eigenfunction pairs in order of increasing roughness of the eigenfunctions, as measured by the RKHS norm.

\begin{algorithm}\label{alg:E} The algorithm assumes that there is an underlying measure-preserving, ergodic flow on $M$ satisfying Assumption~\ref{assump:A1}, and the kernel $k$ satisfies Assumption~\ref{assump:A2}.
\begin{itemize}
\item Input: sampling interval $\Delta t$; the values of a kernel $k$ on a trajectory $ \{x_0,\ldots,x_{N-1}\}$ of length $N$; thresholds $\delta_0, \delta_1>0$; integers $ l_0, l_1 $ such that $0\leq l_0\leq l_1< N$.
\item Output: A collection of approximate Koopman eigenpairs $( \omega_1, \bar f_{\omega_1} ), \ldots, ( \omega_m, \bar f_{\omega_m,} ) $, with $ \omega_j \in \real $, $ \bar f_{\omega_j} \in \RKHS $, and $ \lVert \bar f_{\omega_1} \rVert_{\mathcal{H}} \leq \cdots \leq \lVert \bar f_{\omega_m} \rVert_{\mathcal{H}}$.
\item Steps
\begin{enumerate}
\item Choose a subset $\Omega_N$ of frequencies contained in $[-\pi/\Delta t ,\pi / \Delta t]$. Calculate the quantities $\norm_{N,l}(f_{\omega})$ for $\omega\in\Omega_N$ and $l \in \{ l_0, l_1 \}$. See Proposition \ref{prop:norm_algo_B} \blue{for a particular choice of $\Omega_N$ which makes this calculation faster by utilizing the FFT.} 
\item Select the $\omega$ in $\Omega_N$ for which $\norm_{N,l_0}(f_\omega)> \delta_0$.
\item Of the remaining $\omega$, discard the ones for which the $\ratio(\omega;l_0,l_1,N)$ in \eqref{eqn:def:ratio} is greater than $\delta_1$. 
\item Collect the non-discarded frequencies $ \omega_j $ from Step~3, and compute the Nystr\"om extensions $ \bar f_{\omega_j } = T_N( f_{\omega_j}|X_N ) $ of the corresponding Fourier functions via~\eqref{eqNystromFourier}.
\end{enumerate}
\end{itemize}
\end{algorithm}

Note that if Algorithm~\ref{alg:E} returns a nonempty set of approximate eigenfrequencies and their corresponding eigenfunctions, then the group structure of Koopman point spectra described above can be employed to generate countably infinitely many (approximate) eigenfrequencies and eigenfunctions. In particular, if the point spectrum is finitely generated by $q$ eigenfrequencies and the algorithm produces approximations to $m \geq q $ rationally independent eigenfrequencies, then the full point spectrum of the Koopman group is effectively approximated. Of course, in a practical computational environment there is no way of rigorously verifying the success of this generative scheme, as it is not possible to decide whether any two candidate frequencies represented in floating-point arithmetic truly approximate rationally-independent eigenfrequencies. Nevertheless, the group structure of Koopman spectra allows one to use more stringent values of the spectral resolution and threshold parameters of the algorithm, since rejection of certain true eigenfrequencies/eigenfunctions (i.e., false negatives) can be be compensated by reconstruction using the group structure. Overall, even though the selection criteria in Algorithm~\ref{alg:E} have an element of subjectivity (as with many threshold-based techniques, including conventional DFT-based spectral estimation), and the risk of false positives cannot be completely eliminated, the efficacy of the procedure in robust Koopman spectral estimation is aided by (i) the use of two independent selection criteria based on Theorems~\ref{thm:A} and~\ref{thm:D}; and (ii) the group structure of the Koopman point spectra, allowing one to focus on detection of a limited number of generating eigenfrequencies and eigenfunctions.

\begin{rk*} According to Theorems~\ref{thm:A} and~\ref{thm:D}, the selected frequencies may not be true Koopman eigenfrequencies, being instead shifts of such frequencies by (unique) integer multiples of $2\pi / \Delta t$. While we do not pursue this option here, if two time series at rationally-independent sampling intervals $\Delta t$ are available, one can employ Corollary~\ref{corFreq} to eliminate this aliasing effect.
\end{rk*}

\paragraph{Candidate frequency selection} We choose the finite trial frequency set $ \Omega_N \subset [- \pi/\Delta t, \pi / \Delta t]$ as the standard DFT frequency grid, so that it both gets denser as $N\to \infty$, and also allows fast efficient computation through the use of fast Fourier transforms (FFTs). Specifically, for fixed $N$, assumed odd for simplicity, we define
\begin{equation}\label{eqn:def:omega_r}
\omega_{r} := 2\pi r/N \Delta t, \quad r \in \{ -(N-1)/2,\ldots,(N-1)/2\}.
\end{equation} 
Letting then $\DiscFour_{N} : \cmplx^N\to\cmplx^N$ be the discrete Fourier transform with $ \vec b = \DiscFour_N \vec a $, $ \vec a = ( a_0, \ldots, a_{N-1}) $, $ \vec b = ( - b_{-(N-1)/2}, \ldots, b_{(N-1)/2} ) $, and 
\[ b_j := \frac{1}{N}\sum_{n=0}^{N-1} e^{-2\pi i nj/N} a_n,\]
the following proposition shows how one can utilize the speed of the FFT to compute $\norm_{N,l}(f_{\omega_r})$ for the frequencies $\omega_r$ in \eqref{eqn:def:omega_r}.

\begin{proposition}\label{prop:norm_algo_B}
Let $\Phimatrix{N}$ be the $N\times N$ matrix whose $(n,j)$-th element is $\phi_{N,j}(x_n)$, and $\Lambda_N$ the diagonal matrix with $\lambda_{N,j}$ as the $j$-th diagonal entry. Let also $A_N : = \DiscFour_{N} \Lambda^{-1/2}_N \Phimatrix{N}$, where $\DiscFour_{N}$ operates columnwise. Then, $\norm_N(f_{\omega_r})$ is equal to the squared $\ell^2$ norm of the $r$-th row of $A_N$, indexed such that $ r \in \{ -(N-1)/2, \ldots, (N-1)/2 \} $. Moreover, $\norm_{N,l}(f_{\omega_r})$ is equal to the squared $\ell^2$ norm of the $r$-th row of $A_N$, truncated to the first $l$ entries.
\end{proposition}
\begin{proof} 
It follows from~\eqref{eqn:trunc} that $ \norm_{N,l}(f_{\omega_r}) = \sum_{j=0}^{l-1} \left| \lambda_{N,j}^{-1/2} a_{r,N,j} \right|^2 $, where $ a_{r,N,j} = \langle \phi_{N,j},f_{\omega_r}|X_N \rangle_{\mu_N} $, and $ n \in \{ 0, \ldots, N - 1 \} $. The claim of the Proposition for $ w_{N}(f_{\omega_r}) $ follows from the fact that $\left( \lambda_{N,j}^{-1/2} a_{r,N,j}^* \right)_{r=-N/2}^{N/2}$ is the DFT of the sequence $\left( \lambda_{N,j}^{-1/2} \phi_{N,j}(x_n) \right)_{n=0}^{N-1}$, and the latter is the $j$-th column of $\Lambda^{-1/2}_N \Phimatrix{N}$. The claim for $\norm_{N,l}(f_{\omega_r})$ follows in an analogous manner.
\end{proof}

\paragraph{Parameter selection} We end this section with general guidelines for choosing the parameters in Algorithm~\ref{alg:E}. First, to choose $\delta_0$, note that among all vectors $ f \in L^2(\mu_N)$ with unit norm, the ones having Nystr\"om extensions with minimal $\RKHS$ norm lie entirely in the top eigenspace of $G_N$, corresponding to eigenvalue $\lambda_{N,0}$. Since $ \lVert f_\omega \rVert_{\mu_N} = 1$ for all Fourier functions $f_\omega$, and $w_{N,l_0}(f_\omega) $ is equal to a spectrally truncated squared RKHS norm, this leads to a characteristic scale $ 1 / \lambda_{N,0}$ for $\delta_0$. In other words, $\delta_0 $ is ``small'', and the selection criterion in Step~2 of the algorithm is stringent, if $\delta_0 \lambda_{N,0} \ll 1$, and one can select this parameter from the interval $ [ 0, \lambda_{N,0} ] $. Note that in the case of the normalized kernels $ \hat p_N $ in~\eqref{eqn:def:Markov}, $ \lambda_{N,0} = 1 $ for all $N \in \num$, leading to a universal characteristic scale for $\delta_0$ equal to 1. 

Next, observe that $r(\omega;l_0,l_1,N)$ is a lower bound for the ratio $ \eta $ between the squared norms of the components of $T_N f_\omega$ in the $\mathcal{H}$-subspaces $\spn\{\psi_{N,0},\ldots,\psi_{N,l_0-1} \}^\perp $ and $ \spn\{ \psi_{N,0},\ldots,\psi_{N,l_0-1} \}$; specifically, 
\[ \eta \geq \frac{w_{N,N}(f_\omega) - w_{N,l_0}(f_\omega) }{w_{N,l_0}(f_\omega)} = r(\omega;l_0,l_1,N) + \frac{w_{N,N}(f_\omega)-w_{N,l_1}(f_\omega)}{w_{N,l_0}(f_\omega)} \geq r(\omega;l_0,l_1,N). \] 
Thus, in Step~3 of Algorithm~\ref{alg:E}, we effectively select candidate frequencies $\omega$ for which the relative energy (squared norm) concentration of $T_N f_\omega$ between $ \spn \{ \psi_{N,0}, \ldots, \psi_{N,l_0-1} \} $ and its orthogonal complement is at least equal to $1/\delta_1$. This suggests that a natural scale for $\delta_1$ is equal to 1, so that $\delta_1 \ll 1 $ ($\delta_1 \gg 1$) corresponds to a highly stringent (relaxed) selection criterion. For example, if $ \delta_1 = 1 $, Step~3 selects frequencies whose corresponding RKHS-extended Fourier functions $T_N f_\omega$ have no greater energy in $\spn \{ \psi_{N,0}, \ldots, \psi_{N,l_0} \}^\perp $ than in $ \spn \{ \psi_{N,0},\ldots,\psi_{N,l_0} \} $.

Finally, while there is no universal choice for the scale of $l_0, l_1$, as a rough practical guideline, one typically works with $0 \ll l_0 \ll l_1 \ll N - 1$. Here, the requirement that $l_1 \ll N-1$ is motivated by the fact that the eigenfunctions $\phi_{N,j}$, and thus the squared RKHS norms $w_{N,l_1}(f_\omega)$, with $ j, l_1 \simeq N - 1$, generally exhibit large sampling errors (i.e., sensitivity to the particular trajectory $X_N$ sampled by the data). The requirements that $l_0 \gg 0 $ and $ l_1 \gg l_0 $ are meant to ensure that the selection criteria in Steps~2 and~3 have high discriminating power, i.e., ability for $w_{N,l_0}(f_\omega)$ and $r(\omega;l_0,l_1,N)$ to reach large values, respectively.

In general, in the absence of relevant prior knowledge about the dynamical system and/or the spectrum of $G_N$ that would enable the derivation of more precise guidelines, one can execute Algorithm~\ref{alg:E} for a range of parameter values in the intervals indicated above, and chose eigenfrequencies that are persistently selected over several parameter values. It is important to note that due to the two independent tests employed in Algorithm~\ref{alg:E}, typically there are several parameter choices leading to consistent results for the selected frequencies.


\section{Examples and discussion}\label{sect:examples}

In this section, we apply the methods described in Sections \ref{sect:intro}--\ref{sect:numerics} to ergodic dynamical systems with different types of spectra. The goal is to demonstrate that the results of Theorems \ref{thm:A} and \ref{thm:D}, as implemented through Algorithm~\ref{alg:E}, are effective in identifying Koopman eigenfunctions and eigenfrequencies. 

We consider the following three systems, whose spectra are respectively pure point, continuous (with a trivial eigenfrequency at zero), and mixed, respectively:
\begin{enumerate}
\item A linear quasiperiodic flow $ R_{\alpha_1,\alpha_2 } $ on $\TorusD{2}$, defined as 
\begin{equation}\label{eqn:2D_oscill}
d R_{\alpha_1,\alpha_2}^t(\theta)/dt = (\alpha_1,\alpha_2), \quad \theta = ( \theta_1, \theta_2 ) \in \mathbb{T}^2, \quad \alpha_1=1, \quad \alpha_2=\sqrt{2},
\end{equation}
and observed through the non-injective observation map $F:\TorusD{2}\to\real$ with
\begin{equation}\label{eqn:2D_oscill_obs}
F(\theta_1, \theta_2) = \sin(\theta_1)\cos(\theta_2).
\end{equation}
This system has a pure point Koopman spectrum, consisting of eigenfrequencies of the form $n_1 \alpha_1 + n_2 \alpha_2$ with $n_1,n_2\in\integer$. Because $\alpha_1$ and $ \alpha_2 $ are rationally independent, the set of eigenfrequencies lies dense in $\real$, which makes the problem of numerically distinguishing eigenfrequencies from non-eigenfrequencies non-trivial despite the simplicity of the underlying dynamics.
\item The Lorenz 63 (L63) flow \cite{Lorenz63}, $ \Phi^t_\text{l63} : \real^3 \to \real^3 $, generated by the $C^\infty$ vector field $ \vec V$ with components $ ( V^{(x)}, V^{(y)}, V^{(z)})$ at $(x,y,z)\in\real^3$ given by
\begin{equation}\label{eqn:l63}
V^{(x)} = \sigma(y-x), \quad V^{(y)} = x(\rho-z) -y, \quad V^{(z)} = xy-\beta z,
\end{equation}
where $ \beta = 8/3 $, $\rho = 28$, and $\sigma = 10 $. The system is sampled through the identity map $F:\real^3\to\real^3$, i.e.,
\begin{equation}\label{eqn:l63_obs}
F(x,y,z) := \left( x,y,z \right).
\end{equation}
The L63 flow is known to have a chaotic attractor $ X_\text{l63} \subset\real^3 $ with fractal dimension $\approx 2.06$ \cite{LorentzFract}, supporting a physical invariant measure \cite{Tucker99}, and is also known to be mixing \cite{LuzzattoEtAl05}. That is, there exist no nonzero Koopman eigenfrequencies for this system.
\item A Cartesian product $\Phi^t_\text{l63} \times R_\alpha^t $ of the L63 flow with a linear flow $ R_\alpha^t $ on $S^1$ defined by
\begin{equation}\label{eqn:1D_oscill}
d R_\alpha^t(\theta)/dt = \alpha, \quad \theta \in S^1, \quad \alpha=1.
\end{equation}
This system is observed through a non-invertible map $ F : \real^3\times S^1 \to \real^3 $, which combines the coordinates of the continuous-spectrum subsystem with the rotation, viz.
\begin{equation}\label{eqn:l63_skew_additive}
F(x,\theta) = x + c\left( \sin(\theta), \cos(2\theta), \sin(2\theta) \right), \quad x\in \real^3, \quad \theta \in S^1, \quad c=0.2.
\end{equation}
This system has a mixed spectrum, containing the discrete eigenfrequency spectrum $ \{ j \alpha: j \in \mathbb{ Z} \} $ of the rotation and the continuous spectrum of the L63 flow as subsets. Due to the smallness of the constant $c$ in~\eqref{eqn:l63_skew_additive}, the L63 signal dominates the contribution from the rotation in the observation map $F$.
\end{enumerate}

\paragraph{Methodology} The following steps describe sequentially the entire numerical procedure carried out. 
\begin{enumerate}
\item Numerical trajectories $x_0, x_1, \ldots, x_{N-1} $, with $x_n = \Phi^{n\Delta t}(x_0)$ of length $N$ are generated, using a sampling interval $\Delta t=0.01$ in all cases. In the L63 experiments, we let the system relax towards the attractor, and set $ x_0 $ to a state sampled after a long spinup time (4000 natural time units); that is, we formally assume that $ x_0 $ has converged to the ergodic attractor. We use the \texttt{ode45} solver of Matlab to compute the trajectories. For the three systems $R_{\alpha_1,\alpha_2}^t$, $\Phi_{\text{l63}}^t$, and $\Phi_{\text{l63}}^t \times R_{\omega}^t$, the sample numbers are $ N= \text{40,000} $, 60,000, and 70,000, and the initial points in state space are $x_0 = (0,0)$, $(0,1,1.05)$, and $(0,1,1.05,0)$, respectively.
\item The observation map $F$ described for each system is used to generate the respective time series $ F( x_0 ), F( x_1 ), \ldots, F( x_{N-1} ) $. This dataset forms the basis of all subsequent computations. We perform delay-coordinate maps to construct an injective observation map $F_Q$ from~\eqref{eqFQ}, using $Q=5$, $2$, and $10$ delays for the three systems, respectively.
\item We employ the normalized kernel $\hat p_N$ in~\eqref{eqn:def:Markov}, obtained from the Gaussian kernel in~\eqref{eqKGauss}. The Gaussian kernel bandwidth $ \epsilon $ is determined automatically via the procedure described in \cite{BerryHarlim16,GraphTomog,BerryEtAl15}, which yields $\epsilon= 0.0459$, 0.015, and 0.0804, for the three systems, respectively.
\item The eigenpairs $(\phi_{N,j}, \lambda_{N,j})$ are computed for $j \in \{ 0,\ldots,l_1 \} $ using Matlab's \texttt{eigs} iterative solver. 
\item The spectrally truncated squared RKHS norms $\norm_{N,l}(f_{\omega_r})$ are computed for $l = l_0, l_1$ and the DFT frequencies $ \omega_r $ from~\eqref{eqn:def:omega_r}. We have experimented with different values of the parameters $ l_0$, $l_1$, $\delta_0$ and $\delta_1$ of Algorithm~\ref{alg:E} in the ranges $l_0, l_1 \leq 2000 $ and $\delta_0, \delta_1 \leq 1 $. In what follows, we show results obtained with $ (l_0,l_1) = (100,1000)$, $(100,1000)$, and $(1240,1500)$ for the three systems, respectively. The values of $(\delta_0, \delta_1)$ equal $(0.1,1)$, $(1,1)$, and $(1,1)$ respectively.
\end{enumerate}

\paragraph{Results} Figures~\ref{fig:2D_oscill}--\ref{fig:l63_skew_additive} show results obtained via the procedure described above for the quasiperiodic rotation on $\mathbb{T}^2$, L63 system on $\real^3$, and the mixed-spectrum system on $\real^3\times S^1$, respectively. Additional plots for the mixed-spectrum system, including a comparison with DFT-based spectral estimation, are included in Fig.~\ref{fig:l63_skew_additive_FFT}. To interpret these results, recall that $\norm_{N}$ is the squared RKHS norm of $T_N(f_\omega|X_N)$, while $\norm_{N,l}$ is the squared RKHS norm of $T_N(f_\omega|X_N) $ projected to the subspace spanned by $\psi_{N,0}, \ldots, \psi_{N,l-1}$. That is, $l$ behaves like spectral resolution parameter. According to Theorem \ref{thm:D}, at fixed finite resolution $l=l_0$, $\norm_{N,l_0}(f_\omega)$ converges to $0$ if $\omega $ has no $2\pi q / \Delta t $ translates in $\Omega$ with $q \in \integer$, whereas Theorem~\ref{thm:A} states that at variable resolution $ l = N-1 $, $\norm_N( f_\omega) = \norm_{N,N-1}(f_\omega)$ diverges as $N\to\infty$. Since in practice we are not at a liberty to increase $N$ to test for the asymptotic behavior of $ w_{N,l_0}(f_\omega) $ and $ w_{N}(f_\omega) $, we take advantage of the different nature of these results to identify Koopman eigenfrequencies via a two-step approach based on Theorems~\ref{thm:D} and~\ref{thm:A} (Steps~2 and~3 in Algorithm~\ref{alg:E}, respectively).

\begin{figure}
\centering
\captionsetup{width=\linewidth}
\includegraphics[width=\textwidth]{\FigPath 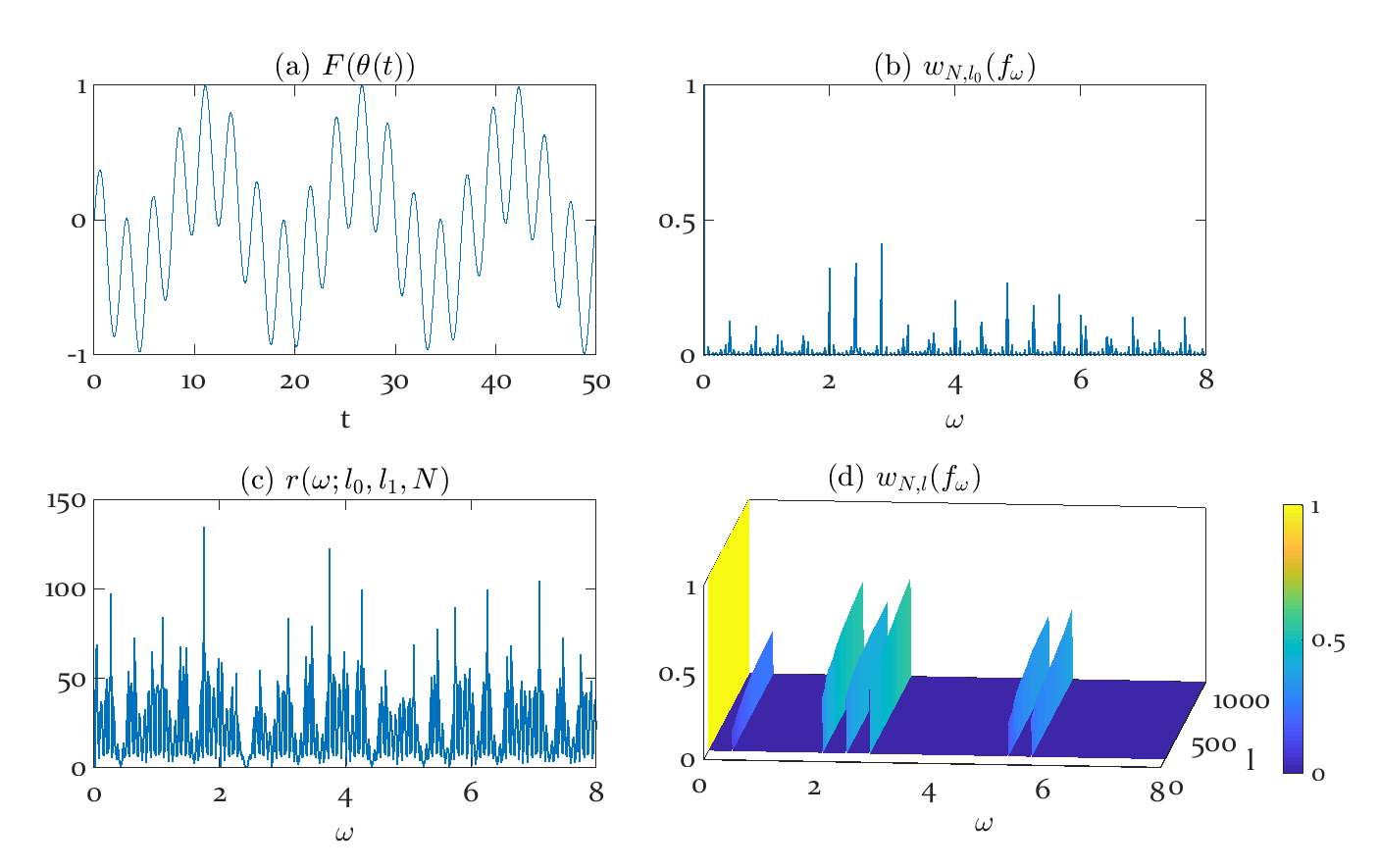}
\caption[results]{ Results of Algorithm \ref{alg:E} applied to the linear quasiperiodic flow \eqref{eqn:2D_oscill} on the 2-torus, using $N=\text{40,000}$ samples. The input is an observable time series, in this case obtained from the map $F$ in \eqref{eqn:2D_oscill_obs} and shown in (a). The goal is to identify Koopman eigenfrequencies of the system. Panel (b) shows the spectrally truncated squared RKHS norm $\norm_{N,l_0}(f_{\omega})$~\eqref{eqn:trunc} as a function of frequency $\omega$ for $l_0=100$. Panel (c) shows the ratio $\ratio(\omega;l_0,l_1,N)$~\eqref{eqn:def:ratio} as a function of $\omega$, computed for $l_1=1000$ and $l_0 $ as in (b). The candidate frequencies $\omega$ for which $\norm_{N,l_0}(f_{\omega}) < \delta_0 $ and $\ratio(\omega;l_0,l_1,N) > \delta_1 $ are discarded, using the threshold values $\delta_0=0.1$ and $\delta_1=1$. The selected frequencies include $2$ and $1+\sqrt{2}$, which are integer linear combinations of the two basic frequencies $1$ and $\sqrt{2}$ of the system. Panel (d) shows the dependence of $ w_{N,l}( f_{\omega} ) $ on $ l $ and $ \omega$ for these selected frequencies (in both vertical surface displacements and colors). Note that the squared norm $w_{N,l}(f_{\omega})$ does not change significantly as $l$ is increased from $l_0$ to $l_1$, and thus the results are robust with respect to the choice of thresholds $\delta_0, \delta_1$. } 
\label{fig:2D_oscill}
\end{figure}

\begin{figure}
\centering
\captionsetup{width=\linewidth} 
\includegraphics[width=\textwidth]{\FigPath 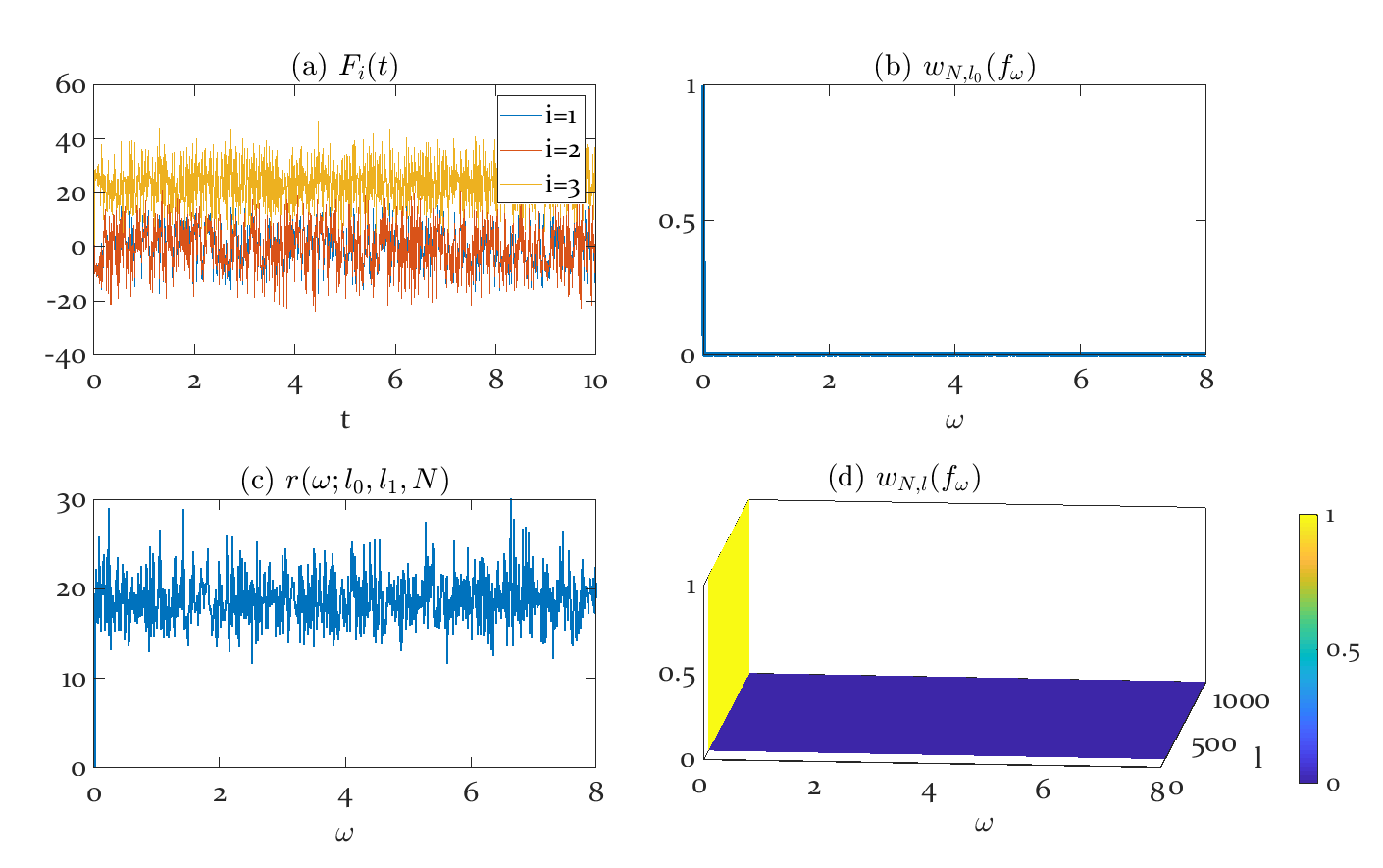}
\caption[results]{ As in Fig.~\ref{fig:2D_oscill}, but for the L63 system in~\eqref{eqn:l63} observed through the observation map in \eqref{eqn:l63_obs}. The components $F_i $ of the input time series are shown in Panel (a). Panel (b) shows that the squared norm $\norm_{N,l_0}(f_\omega)$, computed here for $ N = \text{60,000} $ and $l_0 = 100$, is uniformly small for $\omega\neq 0$, which is consistent with Theorem~\ref{thm:F}. Panel (c) shows the ratio $\ratio(\omega; l_0, l_1, N)$ for $l_1 = 1000$. The large numerical values for $\omega \neq 0 $ are consistent with the fact that for a system without nonzero eigenfrequencies, $\norm_{N,l}(f_\omega)$ grows without bound for all $\omega$ as $l,N\to\infty$ . Panel (d) shows the frequencies that are selected by Algorithm~\ref{alg:E}, with the choice of thresholds $ \delta_0 = \delta_1 = 1 $. The only selected frequency is $ \omega = 0 $, which corresponds to the constant eigenfunction. }%
\label{fig:l63}
\end{figure}

\begin{figure}
\centering
\captionsetup{width=\linewidth}\label{subfig:1b}
\includegraphics[width=\textwidth]{\FigPath 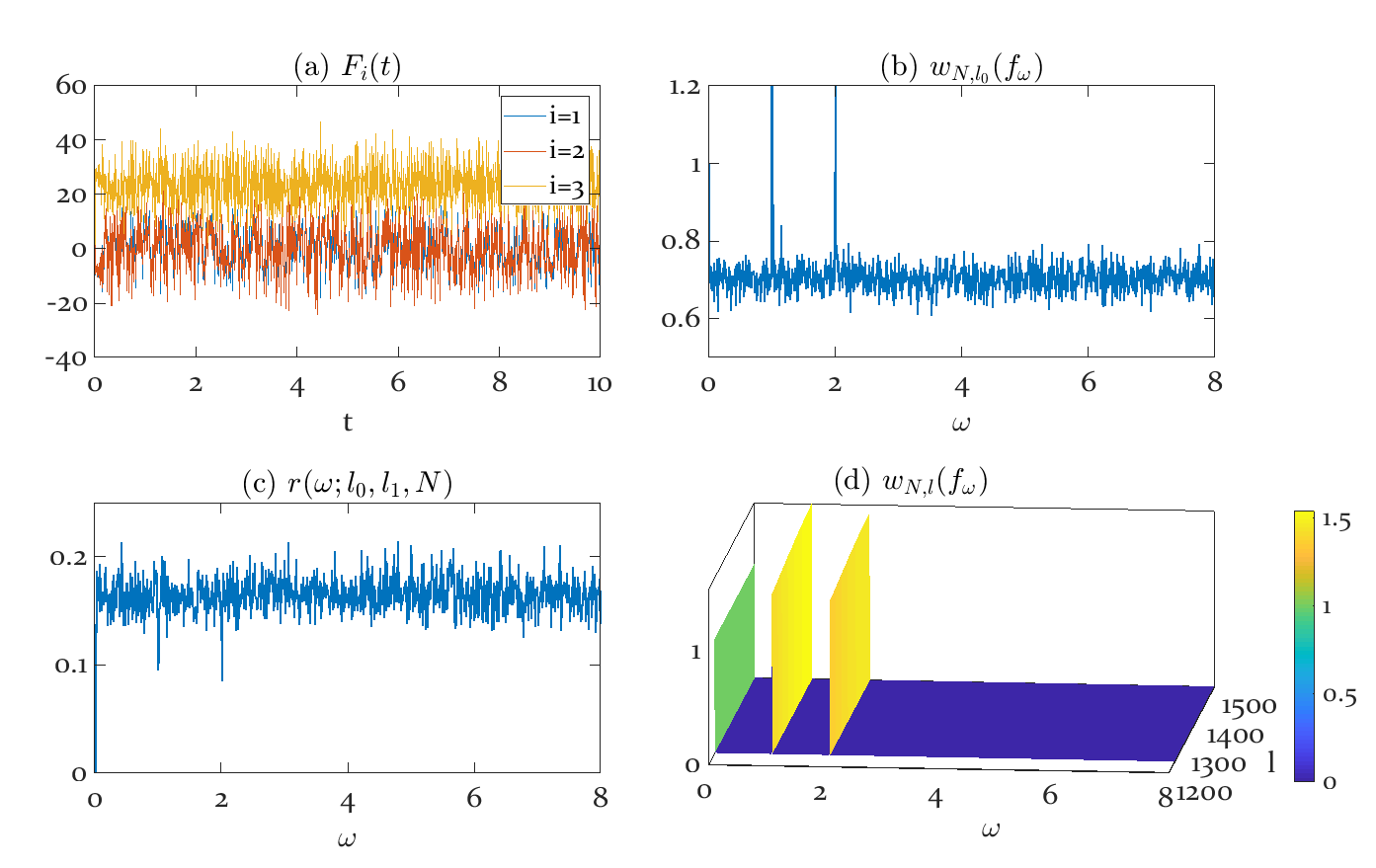}
\caption[results]{ As in Fig.~\ref{fig:2D_oscill}, but for the mixed-spectrum system on $\real^3\times S^1$ from~\eqref{eqn:1D_oscill}, observed via the map $ F $ from~\eqref{eqn:l63_skew_additive}. The number of samples is $N=\text{70,000}$, while Algorithm~\ref{alg:E} is executed using the parameters $ ( l_0, l_1 ) = ( 1240, 1500 ) $ and $ \delta_0 = \delta_1 = 1 $. Panels (a), (b) are reproduced from Fig.~4 for convenience. The Koopman eigenfrequencies identified by the algorithm in (d) are $ \omega = 0 $, 1, and 2, consistent with the discrete spectrum of the rotation $ R^t_\alpha $ with $ \alpha = 1 $.}
\label{fig:l63_skew_additive}
\end{figure}

As can be seen in Figs.~\ref{fig:2D_oscill}--\ref{fig:l63_skew_additive}(d), the method leads to accurate identification of Koopman eigenfrequencies in all three systems studied. In the case of the quasiperiodic rotation on $\mathbb{T}^2$ (Fig.~\ref{fig:2D_oscill}) and the mixed-spectrum system on $\real^3\times S^1 $ (Fig.~\ref{fig:l63_skew_additive}), the number of eigenfrequencies identified is sufficient to generate the full point spectrum via linear combinations. Specifically, in Fig.~\ref{fig:2D_oscill}(d) we find the frequencies 0.419, 2.0, 2.419, 2.827, 5.246, and 5.655, which agree with the theoretically expected eigenfrequencies $\sqrt{2}-1$, $2$, $\sqrt{2}+1$, $2\sqrt{2}$, $3\sqrt{2}-1$, $3\sqrt{2}+1$, and $4\sqrt{2}$ respectively, of the $\mathbb{T}^2$ rotation to within two significant digits. Moreover, in Fig.~\ref{fig:l63_skew_additive}(d), we recover the two eigenfrequencies of the circle rotation, 1 and 2, present in the observation map. In the case of the L63 flow (Fig.~\ref{fig:l63}), the method only identities the trivial (zero) eigenfrequency, consistent with the fact that this system is mixing and its associated Koopman group on $L^2(\mu)$ does not have nonconstant eigenfunctions. In fact, if the L63 flow is assumed to have a Lebesgue absolutely continuous spectrum (which, to our knowledge, has not been shown), then Fig.~\ref{fig:l63}(b) is consistent with Theorem~\ref{thm:F}, according to which, for $\mu$-a.e.~$x_0$ and fixed $l$, the squared norms $\norm_{N,l}(f_\omega)$ converge to $0$ uniformly over $\omega\neq 0$. 

Note that the frequencies identified by Algorithm~\ref{alg:E} are fairly insensitive to the input parameters $ l_0$, $l_1$, $\delta_0 $, and $\delta_1 $, and the two-step approach for selecting eigenfrequencies contributes at least partly to that robustness. In particular, recall that for the Markov kernels in~\eqref{eqn:def:Markov}, a characteristic scale for $\delta_0$ is equal to 1, so that $\delta_0 \ll 1 $ would correspond to a ``stringent'' test in Step~2 of the Algorithm. For the L63 system in Fig.~\ref{fig:l63}, Step~2 (which uses $\delta_0 = 1$) is not stringent as it leads to no rejections of candidate eigenfrequencies. We chose $ \delta_1 = 1$ in Step 3, which corresponds to the requirement of selecting frequencies whose corresponding RKHS-extended Fourier functions have at least half of their squared norm concentrated on the subspace spanned by $ \psi_{N,0}, \ldots, \psi_{N,l_0}$. In spite of the relatively modest strength of this requirement, all nonzero frequencies are discarded as non-eigenfrequencies, as suggested by Theorem~\ref{thm:A}(i). Conversely, in the case of the mixed-spectrum system, Step~2 with $\delta_0 = 1$ accurately identifies two true eigenfrequencies (Fig.~\ref{fig:l63_skew_additive}(b)), but Step~3 (Fig.~\ref{fig:l63_skew_additive}(c)), which uses $ \delta_1= 1$ is superfluous. Moreover, the results of Step~3 remain unchanged by varying $ l_1 $ in the interval $800\lesssim l_1 \lesssim 1100$ (Fig.~\ref{fig:2D_oscill}(d)). Overall, these results demonstrate that there is a degree of redundancy between the two tests based on Theorems~\ref{thm:A} and~\ref{thm:D}, contributing to the overall robustness of the procedure. As shown in Fig.~\ref{fig:2D_oscill}(d), the selected frequencies for the torus rotation are similarly robust to changes of parameter values.

\paragraph{Covariance kernels} One of the requirements for Theorems \ref{thm:A}--\ref{thm:F} to hold is that the kernel is strictly positive-definite. As stated in Section~\ref{sect:numerics}, this requirement is not satisfied when using a covariance kernel $k(x,x') = \langle F(x), F(x') \rangle_{Y}$ from~\eqref{eqKCov} associated with an observation map $ F : M \to Y $ taking values in $ \real^m $. In particular, the rank of the kernel integral operators $ K_N : L^2(\mu_N) \to \RKHS $ associated with such a kernel (and also the rank of $K : L^2(\mu) \to \RKHS $) is at most $m$, meaning that $f_\omega|X_N \in L^2(\mu_N) $ may fail to have a Nystr\"om extension in $ \RKHS$ (for the domain $D(T_N) $ of the extension operator $ T_N $ will be a strict subspace of $ L^2(\mu_N) $ for $ N > \rank K_N $). In effect, a covariance kernel on a finite-dimensional data space significantly limits the richness of observables in the corresponding RKHS, thus decreasing the likelihood that Koopman eigenfunctions can be found in this space.

\paragraph{Comparison with harmonic averaging} To compare our RKHS approach (with $k$ set to a covariance kernel as above) with conventional harmonic averaging, observe that even if $ f_\omega|X_N $ does not lie in $D(T_N)$ , it is still possible to compute the $\RKHS$ extension of the orthogonal projection $g_{\omega,N} \in D(T_N)$ of $f_\omega|X_N$ onto the domain of the Nystr\"om extension operator $T_N$, and evaluate the squared RKHS norm 
\begin{equation}\label{eqRKHS_Cov}
\lVert T_N g_{\omega,N} \rVert^2_{\RKHS} = \norm_{N,l_N}(f_\omega) = \sum_{j=0}^{l_N-1} \left| \langle \phi_{N,j}, f_{\omega}|X_N \rangle_{\mu_N} \right|^2/\lambda_{N,j},
\end{equation}
where $l_N = \rank K_N \leq m $. Note also that the kernel integral operator $ G_N = K_N^* K_N : L^2(\mu_N) \to L^2(\mu_N) $ associated with the covariance kernel takes the form 
\begin{equation}\label{eqGAN}
G_N = A_N^* A_N, 
\end{equation}
where $ A_N : L^2( \mu_N) \to \real^m $ is the rank-$l_N$ operator acting on $f \in L^2(\mu_N) $ by component-wise integration against the observation map, viz.
\begin{equation}\label{eqAN}
A_N f = \int_M F(x) f(x) \, d\mu_N(x).
\end{equation}
It then follows from~\eqref{eqGAN} that the $L^2(\mu_N) $ basis vectors $\phi_{N,0}, \ldots, \phi_{N,l_N-1}$ are also right singular vectors of $ A_N $ corresponding to the (strictly positive) singular values $\lambda_{N,0}^{1/2}, \ldots, \lambda_{N,l_N-1}^{1/2}$, respectively. This property, in conjunction with~\eqref{eqAN}, leads to
\[ F(x_n) = \sum_{j=0}^{l_N-1} e_{N,j} \lambda_{N,j}^{1/2} \phi_{N,j}(x_n), \quad \forall x_n \in X_N, \]
where $e_{N,0}, \ldots, e_{N,l_N-1}$ are orthonormal left singular vectors of $A_N $ in $\real^m$. Inserting this representation of $F(x_n) $ in the harmonic averaging formula in~\eqref{eqn:FFT}, we obtain
\begin{equation}\label{eqHarmAv}
\left\| \Fourier_{\omega,N} F \right\|_{\real^m}^2 = \sum_{j=0}^{l_N-1} \left| \langle \phi_{N,j}, f_{\omega}|X_N \rangle_{\mu_N} \right|^2 \lambda_{N,j} .
\end{equation}

A comparison of~\eqref{eqRKHS_Cov} and~\eqref{eqHarmAv} then shows that the power spectral density $\left\| \Fourier_{\omega,N}F \right\|_{\real^m}^2$ from harmonic averaging has a structurally similar representation to the squared RKHS norm $\lVert T_Ng_{\omega,N} \rVert^2_{\RKHS}$ in terms of the $(\lambda_{N,j}, \phi_{N,j})$ eigenpairs, apart from the fact that the former involves multiplication by $\lambda_{N,j} $ (thus being dominated by the projections of $f_\omega|X_N$ along the most energetic signal components), whereas the latter involves division by $\lambda_{N,j}$ (thus being dominated by the projections of $f_\omega|X_N$ along the most irregular signal components in the sense of the covariance kernel). In applications where the ratio $\lambda_{N,0} / \lambda_{N,l_N-1}$ is not too large, $\left\| \Fourier_{\omega,N}(F) \right\|_{\real^\dimObs}^2$ and $ w_{N,l_N}(f_\omega)$ will thus be comparable. If, however, $\lambda_{N,0} / \lambda_{N,l_N-1}$ is large (as will typically the case in high data space dimensions), then, depending on the frequency $ \omega $, the two quantities can be vastly different. As illustrated in Fig.~\ref{fig:l63_skew_additive_FFT}(d), the limitation $ l_N \leq m $ may be inadequate for estimating eigenfrequencies from observables of mixed-spectrum systems dominated by the continuous spectrum.

\blue{There is also a simple example of an injective embedding $F:M\to Y$ for which a direct DFT fails to yield all the eigenfrequencies. Consider the case where $X = M = S^1$, and $\Phi^t$ is the rotation by a constant velocity $\omega$. Take $Y=\mathbb{R}^4$, and let $F:M\to Y$ be the map $F(x) = \left( \sin(2x), \cos(2x), \sin(3x), \cos(3x) \right)$. Then, $F$ is an injective embedding of $X$ into $Y$, but a simple Fourier analysis of the components of $F$ will only yield the two frequencies $2\omega$ and $3\omega$. Of course, linear integer combinations of these two yield other multiples of $\omega$ in a post processing step, but this would lead to increased error sensitivity in calculating the generating frequencies. On the other hand, the RKHS-based approach here is not limited by the dimensionality of $F$, but depends on the choice of the spectral resolution parameter $l$.}

An additional consideration that should be kept in mind when interpreting the relationship between harmonic averaging and the RKHS-based approach is that for a fixed finite spectral resolution $l \leq m $, Theorem \ref{thm:D} shows convergence results similar to harmonic averaging (that is, if either of $ \lVert \mathcal{ F }_{\omega,N} \rVert^2_{\real^m} $ or $ w_{N,l}( f_\omega ) $ do not converge to zero as $ N \to \infty $, then $ \omega $ is an eigenfrequency), while taking $l=N$, Theorem \ref{thm:A} shows convergence properties of a fundamentally different nature. In effect, Theorem~\ref{thm:A} states that if the projection of $ f_\omega | X_N $ onto the eigenspaces of $ G_N $ corresponding to $ \RKHS $ subspaces of low regularity decays rapidly-enough as $N $ increases, then $ \omega $ is an eigenfrequency. In order for this result to hold, the RKHS $\RKHS$ must lie dense in $ L^2(\mu) $ (which implies that the rank of $ G_N $ increases without bound as $ N \to \infty $, and its smallest eigenvalue $ \lambda_{N,N} $ is strictly positive and converges to zero), and this will not occur with non-positive-definite kernels such as covariance kernels associated with data in $\real^m$. Nevertheless, it is possible that for sufficiently large $ l_N = \rank G_N $ and eigenvalue ratio $ \lambda_{N,0}/ \lambda_{N,l_{N-1}}$ the squared RKHS norms $ \lVert T_N g_{\omega,N} \rVert^2_{\RKHS}$ from~\eqref{eqRKHS_Cov} may approximate the behavior established in Theorem~\ref{thm:A}.

\paragraph{Summary} By working with strictly positive-definite kernels, the RKHS approach described in this paper provides two distinct criteria to identify Koopman eigenfrequencies. The truncated squared RKHS norm $\norm_{N,l}$ in \eqref{eqn:trunc} effectively computes harmonic averages of $l$ observables $\phi_{N,0},\ldots,\phi_{N,l}$, with an important weighting by the inverse of the corresponding eigenvalues $ \lambda_{N,j}$, and as $ N $ grows, this collection provides infinitely many observables to carry out harmonic averaging on. This results in estimates of the RKHS norms of candidate Koopman eigenfunctions, which is a measure of their regularity and an important criterion in identifying eigenfrequencies. In addition, the method computes approximations of everywhere defined (as opposed to $\mu$-a.e.\ defined) Koopman eigenfunctions $ T_N( f_\omega | X_N ) \in \RKHS$, which can be evaluated at arbitrary points in $M$, and converge in $L^2(\mu)$ norm as $N\to\infty$ (Theorem~\ref{thm:D}(ii)). For Koopman eigenfunctions in $L^2(\mu)$ with representatives in $\mathcal{H}$, the convergence was shown to take place in the stronger, RKHS norm (Theorem~\ref{thm:A}), which implies uniform convergence on the support of the invariant measure.

\paragraph{Acknowledgments} Dimitrios Giannakis received support from ONR YIP grant N00014-16-1-2649, NSF grant DMS-1521775, and DARPA grant HR0011-16-C-0116. Suddhasattwa Das is supported as a postdoctoral research fellow from the first grant. The authors are \blue{ also grateful to Corbinian Schlosser for his insightful feedback, and three anonymous referees for making a number of technical suggestions that have helped us improve the paper.}



\bibliography{References}
\end{document}